\documentclass[11pt]{article}

\usepackage[english]{babel}
\usepackage[ansinew]{inputenc}
\usepackage{amsmath, amsthm, amsfonts, amssymb}
\usepackage{latexsym}
\usepackage{bbm}
\usepackage{mathrsfs}
\usepackage{tikz}
\usepackage{graphicx}
\usepackage{multirow}
\usepackage{subfig}
\usepackage{youngtab}

\newtheorem{theorem}{Theorem}[section]
\newtheorem{proposition}[theorem]{Proposition}
\newtheorem{lemma}[theorem]{Lemma}
\newtheorem{corollary}[theorem]{Corollary}
\theoremstyle{definition}
\newtheorem{definition}[theorem]{Definition}

\newtheorem{example}[theorem]{Example}
\newtheorem{problem}[theorem]{Problem}
\newtheorem{conjecture}[theorem]{Conjecture}
\theoremstyle{remark}
\newtheorem{remark}[theorem]{Remark}

\numberwithin{equation}{section}
\numberwithin{theorem}{section}

\def\N{\mathbb{N}}
\def\Z{\mathbb{Z}}

\def\R{\mathbb{R}}
\def\C{\mathbb{C}}

\def\la{\lambda}
\def\a{\alpha}
\def\b{\beta}
\def\g{\gamma}
\def\s{\sigma}
\def\d{\delta}
\def\e{\varepsilon}

\def\Oh{\mathcal{O}}

\def\det{\mathrm{det}}
\def\per{\mathrm{per}}
\def\GL{\mathrm{GL}}

\def\SL{\mathrm{SL}}
\def\Hom{\mathrm{Hom}}

\def\lan{\langle}
\def\ran{\rangle}
\def\ot{\otimes}

\def\id{\mathrm{I}}
\def\ti{\times}

\def\mult{\mathrm{mult}}

\def\diag{\mathrm{diag}}
\def\sgn{\mathrm{\small sgn}}

\newcommand{\ol}[1]{{\overline{#1}}}

\def\Id{{\mathrm{id}}}
\def\Sym{{\mathsf{Sym}}}

\def\stab{{\mathrm{stab}}}

\def\Tc{\mathcal{T}}

\newcommand{\dc}{\mathsf{dc}}
\newcommand{\aS}{\ensuremath{\mathfrak{S}}}
\newcommand{\gO}{\ensuremath{\mathcal{O}}}

\newcommand{\IC}{\ensuremath{\mathbb{C}}}

\def\E{\mathcal{E}}

\def\Ct{\C^{\ti}}
\def\e{\varepsilon}

\newcommand*{\ket}[1]{| #1 \rangle}

\newcommand{\csgn}{\mathit colsgn}
\newcommand{\supp}{\mathrm supp}

\newcommand{\tr}{\mathrm{trace}}
\newcommand{\End}{\mathrm{End}}
\newcommand{\spann}{\mathit{span}}

\title{{\bf Fundamental invariants of orbit closures}}

\author{Peter B\"urgisser\thanks{
Institute of Mathematics, Technische Universit\"at Berlin, 
pbuerg@math.tu-berlin.de.
Partially supported by DFG grant BU 1371/3-2.}
\and 
Christian Ikenmeyer\thanks{
Texas A\&M University, 
ciken@math.tamu.edu} 
}
\date{\today}

\begin{document}
\raggedbottom
\maketitle

\begin{abstract}
For several objects of interest in geometric complexity theory, namely for
the determinant, the permanent, the product of variables, the power sum, the unit tensor, and the matrix multiplication tensor,
we introduce and study a fundamental SL-invariant function that relates the coordinate ring of the orbit with the coordinate ring of its closure. 
For the power sums we can write down this fundamental invariant explicitly in most cases. 
Our constructions generalize the two Aronhold invariants on ternary cubics.
For the other objects we identify the invariant function conditional on intriguing combinatorial problems 
much like the well-known Alon-Tarsi conjecture on Latin squares.
We provide computer calculations in small dimensions for these cases.
As a main tool for our analysis, we determine the stabilizers, and we establish the polystability of all the mentioned forms and tensors
(including the generic ones). 
\end{abstract}

\smallskip


\smallskip


\section{Introduction}

\subsection{Motivation}

In 1979 Valiant showed that for every polynomial $p$
there exists a matrix $A$ whose entries are affine linear forms in the variables of~$p$
such that $\det(A)=p$.
The minimal size of such a matrix is called the \emph{determinantal complexity} $\dc(p)$.
The size $\dc(p)$ is closely related to the number of computation gates in arithmetic circuits computing~$p$, 
see \cite{Val:79b, toda:92, mapo:08} for more information.
Valiant's flagship conjecture is concerned with the determinantal complexity of the permanent polynomial:
\[
\per_m := \sum_{\pi\in S_m} X_{1,\pi(1)} X_{2,\pi(2)} \cdots X_{m,\pi(m)},
\]
a homogeneous degree $m$ polynomial in $m^2$ variables.
\begin{conjecture}
The sequence $\dc(\per_m)$ grows superpolynomially fast.
\end{conjecture}
To study this conjecture, Mulmuley and Sohoni \cite{gct1, gct2} proposed an approach through algebraic geometry and representation theory,
for which they coined the term \emph{geometric complexity theory}.
The main idea is to study the $\GL_{n^2}$-orbit closure $\overline{\GL_{n^2}\det_n}$ of the determinant polynomial
\[
\det_n := \sum_{\pi\in\aS_n} \sgn(\pi) X_{1,\pi(1)} X_{2,\pi(2)} \cdots X_{n,\pi(n)}
\]
in the space $\Sym^n\C^{n^2}$ 
of all homogeneous degree $n$ polynomials in $n^2$ variables.
For $n>m$ we compare this closure to the $\GL_{n^2}$-orbit closure $\overline{\GL_{n^2}\per_{n,m}}$ of the \emph{padded} permanent
$\per_{n,m} := (X_{n,n})^{n-m}\per_m$.
We consider the $\GL_{n^2}$-action on the coordinate rings of these closures and note the following:
If there exists an irreducible representation in $\gO(\overline{\GL_{n^2}\per_{n,m}})$
that does not lie in $\gO(\overline{\GL_{n^2}\det_n})$, then $\dc(\per_m)>n$.

A main approach towards understanding $\gO(\overline{\GL_{n^2}\det_n})$
is by studying it as a subalgebra of the coordinate ring $\gO(\GL_{n^2}\det_n)$ of the orbit $\GL_{n^2}\det_n$.
The representation theoretic decomposition of $\gO(\GL_{n^2}\det_n)$ can
be deduced rather explicitly using the algebraic Peter-Weyl theorem \cite[\S4 and \S5]{BLMW:11}.
The main motivation of this paper is to improve our understanding of the connection between 
the coordinate ring of the orbit $\GL_{n^2}\det_n$ and its closure $\ol{\GL_{n^2}\det_n}$. 

\subsection{Results}

We consider the space $\Sym^D\C^m$ of (homogeneous) forms of degree $D$ in $m$ variables 
with the natural action of the group $\GL_m$. 
Let $w\in\Sym^D\C^m$ be a polystable form, which means that its $\SL_m$-orbit is closed. 
We define the {\em stabilizer period} $a(w)$ as the order of the subgroup of $\Ct$ 
generated by $\det(g)$, where $g$ runs through the stabilizer group $\stab(w)$ of $w$. 
In order to study  in which degrees~$d$ there exist $\SL_m$-invariants of $\Oh(\ol{Gw})$, 
we study the {\em degree monoid} of $w$:
$$
 \E(w) := \big\{d\in\N \mid \Oh(\ol{Gw})^{\SL_m}_d\ne 0 \big\} .
$$
We note that $\Oh(\ol{Gw})^{\SL_m}_d$ is at most one-dimensional for all $d\in\Z$.  
We prove that $\E(w)$ generates the group $b(w)\Z$ 
where $b(w)=\frac{m}{D}a(w)$, which we call the {\em degree period} of $w$ 
(Theorem~\ref{th:group_S_w}). 
A more geometric characterization of the degree monoid $\E(w)$ is as follows. 
Consider the well-defined map $\phi_w\colon Gw\to\C,\, g\mapsto \det(g)^{a(w)}$. 
We can show that the zero extension of $\phi_w$ to the boundary $\ol{Gw}\setminus Gw$ 
is continuous in the Euclidean topology. 
However, this extension might not be a regular function. 
In fact, for $k\in\N$, the $k$th power $(\phi_w)^k$ has a regular extension to $\ol{Gw}$ 
iff $b(w)k \in\E(w)$.  

Our primary object of interest is 
the {\em fundamental invariant} $\Phi_w$ of~$w\in\Sym^D\C^m$,
which we define as the 
$\SL_m$-invariant of minimal degree which occurs in $\Oh(\ol{Gw})$
(requiring $\Phi_w(w)=1$ for uniqueness). 
Its degree $e(w)$ is the minimal positive element of $\E(w)$. 

The boundary $\ol{Gw}\setminus Gw$ in $\ol{Gw}$
is the zero set of $\Phi_w$ in $\Oh(\ol{Gw})$ 
and hence $\Oh(\ol{Gw})$ is the localization of $\Oh(Gw)$ 
with respect to $\Phi_w$ (Proposition~\ref{le:ext-zero}). 
However, this holds only set-theoretically in most situations of interest.
More specifically, we prove that if $b(w) < e(w)$, then 
the vanishing ideal of the boundary $\ol{Gw}\setminus Gw$ in $\ol{Gw}$
is strictly larger than the principal ideal $\Phi_w\Oh(\ol{Gw})$.
Moreover, $\ol{Gw}$ is not a normal algebraic variety
(see Theorem~\ref{le:Phi_w-zeroset}). 
The non-normality of $\ol{\GL_{n^2}\det_n}$ 
(or even $\ol{\GL_{m} X_1\cdots X_m}$) 
is the main source of difficulty in implementing 
Mulmuley and Sohoni's strategy, see \cite{BI:10, Kum:15,BHI:15}.

A result due to Howe~\cite{howe:87} implies that 
for generic forms $w\in\Sym^D\C^m$, with even $D$, 
the minimal degree of $\E(w)$ equals $m$.  
The corresponding fundamental invariant is obtained 
as the restriction of a beautiful $\SL_m$-invariant 
$P_{D,m}\colon \Sym^D \C^m \to \C$,  
for which we found an explicit expression; see~\eqref{eq:defP}.
We learned that this invariant already had been discovered by Cayley in 1843
(``On the theory of determinants", Trans. Cambridge Phil Soc. VIII (1843): pp.~1--16). 
We call $P_{D,m}$ the {\em generic fundamental invariant} of $\Sym^D\C^m$. 
It is an irreducible polynomial of degree~$m$.
If  $D=m$ is odd we show that for generic forms $w\in\Sym^D\C^m$,  
the minimal degree of $\E(w)$ equals $m+1$ 
and we explicitly construct this invariant 
in Section~\ref{subsubsec:Dodd}, cf.~\eqref{eq:def:PD}.  
This generalizes the degree~4 Aronhold invariant on ternary cubics 
(case $D=m=3$); see also \cite[eq.~(80)]{howe:87}.

In Section~\ref{se:mindegforms} we investigate the fundamental invariants 
for specific forms $w$ that are of interest for geometric complexity theory. 
If $P_{D,m}(w) \ne 0$, then the fundamental invariant~$\Phi_w$ 
is obtained from the generic fundamental invariant $P_{D,m}$ by restriction 
(and rescaling). It is therefore important to decide whether $P_{D,m}(w) \ne 0$. 
We prove that this is the case for the power sums 
$X_1^D + \cdots + X_m^D$ of even degree~$D$. 
For odd $D$ the minimal degree of the power sum equals $2m$ and 
we explicitly construct its fundamental invariant, which generalizes
Aronhold's construction of the degree 6 invariant on ternary cubics
(which we obtain as the special case $D=m=3$).
Quite interestingly, this construction only works as long as $2m \leq \binom{2D}D$,
whereas for $2m > \binom{2D}D$, the invariant does not occur in any degree up to $2m$, 
see Proposition~\ref{pro:evenodd}.
For other motivation from complexity theory, we note that 
the $\GL_m$-orbit closure of the power sum $X_1^D + X_2^D + \cdots + X_m^D$ is closely related to 
$\Sigma\Lambda\Sigma$ circuit complexity \cite{GKKS:13} 
and to the polynomial Waring problem \cite[\S5.4]{Lan:11}. 

The product of variables $X_1\cdots X_m$ is of interest in geometric complexity theory because it serves 
as a simpler version of the determinant that is still interesting enough to deduce nontrivial facts about the 
determinant \cite{Kum:15, BHI:15}. However, already for $w=X_1\cdots X_m$ (and $m$ even), 
deciding $P_{m,m}(w) \ne 0$ turns out to be a difficult question: 
it is equivalent to the Alon-Tarsi conjecture~\cite{AT:92} on Latin squares, 
which has so far only been verified in special cases.  
(The conjecture is known to be true if $n=p\pm 1$ where $p$ is a prime, 
cf.~\cite{Dri:98,Gly:10}.)
For odd $m$ we formulate an analogous combinatorial conjecture for deciding $P_{m,m}(w)\neq 0$. 

For geometric complexity theory, it would be important to know whether 
$P_{n,n^2}$ does not vanish on the determinant $\det_n$. We rephrase  
this as a combinatorial condition (Proposition~\ref{re:eval-P-det-per})
but only managed to verify it for 
$n=2,4$ by computer calculations. 

The second part of the paper is devoted to analogous investigations 
for cubic tensors $w\in\ot^3\C^m$ with the action of the group
$\GL_m\times \GL_m\times \GL_m$. 
These investigations are motivated by border rank and 
the complexity of matrix multiplication (see \cite{ACT} for an overview), 
which can also be approached with geometric complexity theory~\cite{BI:10,BI:13}. 

The general theory is analogous to the case of forms.
(It is obvious that everything could be developed in more generality, 
but we refrain from doing so.) 
We assign to a polystable tensor $w\in\ot^3\C^m$ its exponent monoid $\E(w)$
and show that $\E(w)$ generates the group $ma(w)$, 
where the stabilizer period $a(w)$ is the order of the subgroup of $\Ct$ 
generated by $\det(g_1)\det(g_2)\det(g_3)$, where 
$(g_1,g_2,g_3)$ runs through the stabilizer group $\stab(w)$ of $w$ 
(Theorem~\ref{th:group_S_w-T}).  
We define the fundamental invariant $\Phi_w$ of $w$ and prove that 
$\Oh(\ol{Gw})=\Oh(Gw)_{\Phi_w}$. 

Suppose now that $m=n^2$. 
In Section~\ref{se:gen-fun-ten}
we show that for generic tensors $w\in\ot^3\C^m$,  
the minimal degree of $\E(w)$ equals $n^3$. 
In fact, for generic $w$, the monoid 
$\E(w)$ is characterized in terms of the positivity 
of Kronecker coefficients of three rectangular shapes, 
see \eqref{eq:kmd}.
The fundamental invariant of  $w\in\ot^3\C^m$ 
is obtained as the restriction of a fascinating $\SL^3_m$-invariant 
$F_{n}\colon \ot^3 \C^m \to \C$, 
for which we found an explicit expression; see~\eqref{eq:fund-invar-tensor}. 
$F_n$ is an irreducible polynomial of degree~$n^3$. 

For understanding border rank, it would be important to verify that 
$F_n$ does not vanish on the unit tensor $\lan n^2 \ran$. 
We characterize this by a combinatorial condition, 
involving ``Latin cubes'', which can be seen as a 3D version 
of the Alon-Tarsi conjecture (Proposition~\ref{le:Fn-unit-tensor}).  
However, so far we verified it only in the cases $n=2,4$ by computer calculations.
We also express the nonvanishing of $F_n$ on the matrix multiplication 
tensor $\lan n,n,n\ran$ by a combinatorial condition 
(Proposition~\ref{le:eval-Fn-mamu}). 

While developing the general theory, we determine the stabilizers of the forms and tensors listed above 
(including the generic ones) and use the Hilbert-Mumford criterion 
(as refined by Luna~\cite{luna:75} and Kempf~\cite{kempf:78})
to prove the polystability of all objects under consideration. 
Some care has to be taken in the first few low-dimensional cases since irregular phenomena appear here. 
As a corollary of our observations we conclude the nonnormality of the orbit closures in most of the cases considered.

\subsection{Organisation of the paper}

Section~\ref{se:prelim} is devoted to determining the stabilizers and periods for 
certain forms and proving their polystability. 
In Section~\ref{se:exp-monoid-forms} we introduce the key concepts of 
the degree monoid and fundamental invariant for polystable forms, but 
postpone the proof of the main results to Section~\ref{se:pf-group_S_w}.  
Furthermore we study the generic fundamental invariant $P_{D,m}$ 
for forms and we investigate their vanishing  
at specific forms, arriving at intriguing combinatorial questions. 

Section~\ref{se:prelim-T} provides the necessary background on stabilizers and polystability 
for tensors. Section~\ref{se:FI-tensors}
is devoted the analysis of the 
degree monoid and fundamental invariants for tensors. 
The remaining proofs are provided in Section~\ref{se:pf-group_S_w}. 
Finally, a short appendix collects information on the stabilizers of 
generic forms in special formats.

\medskip

\noindent{\bf Acknowledgments.} 
We thank 
Matthias Christandl, 
Jesko H\"uttenhain, 
Neeraj Kayal,
Shrawan Kumar,
Joseph Landsberg, 
Laurent Manivel, 
Giorgio Ottaviani, 
Vladimir Popov,
Michael Walter,
and
Jerzy Weyman 
for helpful discussions. 
We are grateful to the Simons Institute for the Theory of Computing for hospitality and financial support during 
the program ``Algorithms and Complexity in Algebraic Geometry'', where part of this work was carried out. 
The calculations were done on the Texas A\&M Calclab Cluster.

\section{Stabilizer period and polystability of forms} 
\label{se:prelim}
\subsection{Stabilizer and stabilizer period of forms}
\label{se:stab-period}

We consider $\Sym^D\C^m$ with the natural action of $\GL_m$. 
The {\em stabilizer} $\stab(w)$ of a form $w\in\Sym^D\C^m$ is 
defined as the following closed subgroup of~$\GL_m$: 
\begin{equation}\label{eq:def-stab}
 \stab(w) := \big\{ g \in \GL_m \mid g w = w \big\} .
\end{equation}
It is clear that 
$\{\zeta \id_m \mid \zeta^D =1 \}$ is contained in $\stab(w)$. 
We say that $w$ has a {\em trivial stabilizer} if equality holds. 

The image $H$ of $\stab(w)$ under the determinant homomorphism 
$\det\colon\GL_m \to \Ct$ is a closed subgroup of $\Ct$ 
(cf.~\cite[\S7.4, Prop.~B]{hump:75}). 
Hence either $H=\Ct$ or $H$ is a finite subgroup 
of the group $\Ct$ and hence cyclic. 
We define the {\em stabilizer period} $a(w)$ as the order of $H$ if 
$H$ is finite and put $a(w):=\infty$ otherwise. 

\begin{lemma}\label{le:reduced-period}
The stabilizer period $a(w)$ is a multiple of $D/\gcd(D,m)$ if $a(w) <\infty$.  
\end{lemma}

\begin{proof}
$H:=\det(\stab(w))$ contains the subgroup
$H' := \{ \zeta^m \mid \zeta^D =1 \}$ 
since $\{\zeta \id_m \mid \zeta^D =1 \}$ is contained in $\stab(w)$. 
But $H'$ equals the group of $D'$th roots of unity, where 
$D':= D/\gcd(D,m)$.
\end{proof}

\begin{definition}\label{def:period}
The {\em reduced stabilizer period} $a'(w)$  of a form $w\in\Sym^D\C^m$ is defined as 
$a'(w) := a(w)\gcd(D,m)/D$, where the {\em stabilizer period} $a(w)$ is defined as the 
order of the subgroup $\det(\stab(w))$. 
\end{definition}

It follows from general principles that 
there exists $a(D,m)$ such that almost all $w\in\Sym^D\C^m$ 
have the stabilizer period $a(D,m)$; 
see~\cite{richardson:72} or Popov and Vinberg~\cite[\S7]{AG4}. 
The {\em generic reduced stabilizer period} $a'(D,m)$,
defined by 
$a'(D,m) := a(D,m)\gcd(D,m)/D$, 
can be explicitly determined.  
Matsumura and Monsky~\cite{mat-mons:63} showed that 
if $D >2$ and $m>3$, then almost all $w\in\Sym^D\C^m$ have 
a trivial stabilizer, which implies $a'(D,m)=1$. 
The following result is a consequence of this 
and a study of the stabilizers of generic binary and ternary forms.
We indicate the proof in Appendix~\ref{appendix}. 

\begin{theorem}\label{th:aDm}
Let $D>2$ and $m\ge 1$. 
The set of $w\in\Sym^D\C^m$ with finite stabilizer is a nonempty open subset of~$W$.
We have $a'(D,m)=1$ except in the following cases: 
$a'(3,2)=2$, 
$a'(3,3)=2$,
$a'(4,3)=2$.
\end{theorem}

We remark that in the exceptional case $D=2$ 
of quadratic forms, the orbit of $X_1^2+\cdots + X_m^2$ is dense in $\Sym^2\C^m$ and 
the stabilizer of $X_1^2+\cdots + X_m^2$ equals the complex orthogonal group, 
which is not finite. We have $a(2,m)=2$. 

In the following we shall exclude the case $D=1$, where the orbit of a nonzero
element $w$ equals $\Sym^1\C^m\setminus\{0\}$ and hence $w$ has the 
stabilizer period $\infty$. 

The following result is well-known
(compare~\cite[Chap.~2]{chen-kayal-wigderson:10} for a proof of the second assertion). 

\begin{proposition}\label{pro:stab-X1-Xn-PS}
\begin{enumerate}
\item The stabilizer of $X_1\cdots X_m$ is generated by the permutation matrices and 
the diagonal matrices with determinant one. 
The stabilizer period of $X_1\cdots X_m$ equals~$2$ if $m\ge 2$.
\item 
Suppose $D>2$ and $m>1$. 
The stabilizer of the power sum $X_1^D+\cdots+X_m^D$ is the subgroup generated by 
the permutation matrices and the diagonal matrices $\diag(t_1,\ldots,t_m)$, 
where $t_1^D=\cdots=t_m^D=1$. 
The stabilizer period of $X_1^D+\cdots+X_m^D$ equals $D$ if $D$ is even and $2D$ if $D$ is odd.
\end{enumerate}
\end{proposition}

We study now the stabilizers and periods of the permanents and determinants. 
We note that the linear transformation 
$g_{A,B}\colon\C^{n\times n} \to \C^{n\times n},\, X \mapsto A X B$, 
defined for $A,B\in\GL_n$, satisfies $\det g_{A,B} = \det(A)^n \det(B)^n$.

\begin{theorem}[Frobenius]\label{th:stab-det}
The stabilizer of $\det_n$ is generated by the transposition~$\tau$  
and the linear transformations $g_{A,B}$, where $A,B\in\SL_n$. 
The stabilizer period of $\det_n$ equals~$1$ if $n\bmod 4 \in \{0,1\}$ and 
the stabilizer period of $\det_n$ equals~$2$ if $n\bmod 4 \in \{2,3\}$. 
\end{theorem}

\begin{proof}
The stabilizer of $\det_n$ has been determined by Frobenius~\cite{Frobdet}
and, independently, by Dieudonn\'e~\cite{dieu:49}.
To determine the period, note that 
$\det g_{A,B} = \det(A)^n \det(B)^n =1$ for $A,B\in\SL_n$. 
Moreover, $\det(\tau) = (-1)^{\frac{n(n-1)}{2}}$. 
Taking into account that $n(n-1)/2$ is even iff $n\bmod 4 \in \{0,1\}$, 
the assertion about the period follows. 
\end{proof}

\begin{theorem}[Marcus and May]\label{th:stab-per}
Let $n>2$. The stabilizer of $\per_n$ is generated by the transposition $\tau$,
and the linear transformations $g_{A,B}$, where $A$ and $B$ are products of 
a permutation matrix with a diagonal matrix of determinant~$1$. 
The stabilizer period of $\per_n$ equals~$2$ unless $n\bmod 4 =0$, in which 
case the stabilizer period equals~$1$.
\end{theorem}

\begin{proof}
The stabilizer of $\per_n$ has been determined by 
Marcus and May~\cite{marcus-may:62}. 
If $A=P_\pi A_1$ and $B=P_\s B_1$ with permutation matrices $P_\pi,P_\s$ and 
diagonal matrices $A_1,B_1\in\SL_n$, then $\det g_{A,B} =\sgn(\pi)^n\sgn(\s)^n$. 
Moreover, recall that $\det(\tau) = (-1)^{\frac{n(n-1)}{2}}$. 
This implies that $\det(\stab(\per_n) )= \{-1,1\}$ iff $n\bmod 4 \neq 0$.
\end{proof}

\subsection{Polystability of forms}
\label{se:polystab}

The concept of polystability is crucial in geometric invariant theory. 
We review the basic facts we need. 

\begin{definition}\label{def:polystab}
A form $w\in\Sym^D\C^m$ is called {\em polystable} iff the $\SL_m$-orbit of $w$ is closed. 
\end{definition}

We present a convenient test for polystability that is based on a refinement of the 
Hilbert-Mumford criterion due to Luna~\cite{luna:75} and Kempf~\cite{kempf:78}. 
The {\em support} of the form $w\in\Sym^D\C^m$ is defined as 
$\supp(w) := \big\{ \a \in\N^m \mid w_\a \ne 0 \big\}$,
where $w=\sum_\a w_\a X^\a$ and $X^\a:= X_1^{\a_1}\cdots X_m^{\a_m}$. 
(Clearly, this is a coordinate dependent notion.) 

\begin{proposition}\label{pro:stab-crit-forms}
Let the form $w\in \Sym^D \C^m$ satisfy the following two properties:
\begin{enumerate}
\item There is a reductive subgroup $R$ of $\SL_m\cap \stab(w)$ 
such that the centralizer of $R$ in $\SL_m$ is contained in the group of diagonal matrices.
\item The convex cone generated by $\supp(w)$ contains $(1,\ldots,1)$.
\end{enumerate}
Then $w$ is polystable.
\end{proposition}

\begin{proof}
Suppose that $w$ is not polystable.  Let $Y$ be an $\SL_m$-orbit in 
$\ol{\SL_m w}\setminus \SL_m w$ of minimal dimension. 
Then $Y$ must be closed. 
The Hilbert-Mumford criterion states that 
there exists a one-parameter subgroup
$\s\colon\C^\ti\to \SL_m$ such that
$y:=\lim_{t\to 0}\s(t)w \in Y$. 
Recall that $R$ is a reductive subgroup of $\SL_m\cap \stab(w)$.  
A result due to Luna~\cite[Cor.~1]{luna:75} 
and Kempf~\cite[Cor.~4.5]{kempf:78} states that
$\s$ may be chosen such that the image of $\s$ is contained 
in the centralizer of $R$ in $\SL_m$. 
Hence, by our first assumption, $\s$ maps $\Ct$ to the 
diagonal matrices, say
$\s(t) = \diag(t^{\mu_1},\ldots,t^{\mu_m})$ for $t\in\Ct$. 
We have $\sum_i \mu_i = 0$ since $\det\s(t) = 1$. 
Note that 
\begin{equation}\label{eq:1-pug-forms}
 \s(t) w = \sum_\a w_\a (t^{\mu_1}X_1)^{\a_1}\cdots (t^{\mu_n}X_n)^{\a_n}
    = \sum_\a t^{\langle \a, \mu\rangle} w_\a X^\a ,
\end{equation}
where 
$\langle \a,\mu \rangle := \sum_i \a_i \mu_i$. 
The existence of $\lim_{t\to 0} \s(t) w$ implies that 
$\lan \a, \mu\ran \ge 0$ for all $\a\in\supp(w)$. 
Our second assumption states that there are $c_\a\ge 0$ 
such that $(1,\ldots,1) = \sum_\a c_\a \a$. 
Hence 
$\sum_i \mu_i = \lan (1,\ldots,1), \mu \ran \ge 0$. 
On the other hand, $\sum_i \mu_i = 0$. 
This implies $\mu_i=0$ for all~$i$ and hence $s(t)=1$ for all $t$. 
Therefore $y=w$, which contradicts the fact the $y$ lies in the 
boundary of the $\SL_m$-orbit of~$w$.
\end{proof}

\begin{corollary}\label{cor:stable-forms}
The forms 
$X_1\cdots X_m$, $X_1^D +\cdots + X_m^D$ if $D>1$, $\det_n$, and $\per_n$
are all polystable. 
\end{corollary}

\begin{proof}
We use Proposition~\ref{pro:stab-crit-forms}. 
For $w=X_1\cdots X_m$ we take for $R$ the group of diagonal matrices with determinant~$1$. 

For $X_1^D +\cdots X_m^D$ we take for $R$ the group of diagonal matrices whose entries are 
$D$th roots of unity.  Let $g$ be in the centralizer of $R$. 
If $D>1$, then for all $i<j$ there is $r\in R$ such that $r_i \ne r_j$. 
Hence $rg=gr$, which implies $r_i g_{ij} = g_{ij} r_j$,  hence $g_{ij}=0$. 
So $g$ must be diagonal and the first property is satisfied. 

For $\det_n$ and $\per_n$ we take for $R$ the group of diagonal matrices in $\SL_{n^2}$. 
The support of both $\det_n$ and $\per_n$ consists of the permutation matrices. 
Summing up the $n$ permutation matrices of the cyclic shifts we get the matrix with only 1s,
which provides the second property in Proposition~\ref{pro:stab-crit-forms}. 
\end{proof}

\begin{proposition}\label{pro:gen-form-stable}
If $D>1$, then almost all $w\in\Sym^D\C^m$ are polystable.
\end{proposition}

\begin{proof}
The case $D=2$ is settled by showing that all $\SL_m$-orbits of  
$\Sym^2\C^m$ are closed. Indeed, if we interpret $\Sym^2\C^m$ 
as the space of symmetric matrices of size $m$, it is straighforward 
to check that 
$\SL_m A = \{B\in \Sym^2\C^m\mid \det B=\det A\}$, 
which is closed. 

Now assume $D>2$. Theorem~\ref{th:aDm} tells us that $\stab(w)$ is finite 
for almost all $w\in\Sym^D\C^m$. 
A general result due to Luna~\cite{luna:73}, 
see also \cite[II 4.3D, Folgerung, p.~142]{kraf:84}, 
implies that $\SL_m w$ is closed for almost all~$w$ 
if $\SL_m \cap \stab(w)$ is finite for almost all~$w$. 
(This general result only requires that $\SL_m$ is a semisimple group.) 
Hence, almost all $w\in\Sym^D\C^m$ are polystable.
\end{proof}

\begin{proposition}\label{pro:polystab-period}
A polystable nonzero form has a finite stabilizer period. 
\end{proposition}

\begin{proof}
Suppose that $w\in\Sym^D\C^m$ has infinite stabilizer period, that is,
$\det(\stab(w))=\Ct$. Then there exists a sequence $g_n\in\stab(w)$ 
such that $\lim_{n\to\infty}|\det(g_n)| = \infty$.
We decompose $g_n=\iota(t_n) h_n$, where $h_n\in\SL_m$ and 
$\iota(t) := t\id_m$. Then $\det(g_n) = t_n^m$. 
We have 
$w = g_n w = \iota(t_n) h_n w = t_n^m h_n w$, 
hence 
$t_n^{-m} w = h_n w \in \SL_m w$.
Taking the limit $n\to\infty$, we conclude that 
$0\in\overline{\SL_m w}$, hence $\SL_m w$ is not closed or $w=0$. 
\end{proof}

\subsection{Numerical semigroups}\label{se:num-sg}

We provide some useful background, see Alfons\'in~\cite{alfonsin:05}. 
A submonoid $S$ of $\N$ is a subset that is closed under addition and contains~$0$. 
A {\em numerical semigroup} $S$ is defined as a submonoid of $\N$ that generates the group $\Z$. 
The submonoid of $\N$ generated by the positive integers $a_1,\ldots,a_n$ consists of 
the nonnegative integer linear combinations $x_1a_1+\cdots +x_na_n$, where $x_i\in\N$. 
Clearly, $S$ is a numerical semigroup iff $\gcd(a_1,\ldots,a_n)=1$.  
It is well-known that any submonoid $S$ of $\N$ is finitely generated. 

Assume that $S$ is a numerical semigroup.
It is a well known fact that $\N\setminus S$ is a finite set. 
The  elements of $\N\setminus S$ are called the {\em gaps} 
(or {\em holes}) of $S$. 
For example, the monoid $S$ generated by $2$ and $5$ is a numerical semigroup with the gaps $1,3$.
The largest gap of~$S$ is called the {\em Frobenius number} $g(S)$ of $S$.

\begin{remark}
If $S$ is generated by $a_1,a_2$, then Sylvester showed that 
$g(a_1,a_2) = a_1a_2 -a_1 -a_2$.
For $n=3$, the formula for $g(a_1,a_2,a_3)$ is quite involved. 
In fact, the structure of numerical semigroups is intricate. 
It is known that the computation of $g(a_1,\ldots,a_n)$ is an NP-hard problem. 
\end{remark}

\section{Fundamental invariant of forms}
\label{se:exp-monoid-forms}

Throughout this section, we assume that $w\in\Sym^D\C^m$ is a polystable form.
It will be convenient to write $G:=\GL_m$ in the following. 

We denote by $\Oh(Gw)^{\SL_m}$ the ring of $\SL_m$-invariants 
of $\Oh(Gw)$ and, for $d\in\Z$, 
by $\Oh(Gw)_d$ its degree $d$ part.

\begin{definition}\label{def:degree-period}
The {\em degree period} $b(w)$ of $w\in\Sym^D\C^m$ is defined as 
$b(w):=\frac{m}{D}a(w)$.  
\end{definition}

The reason for the naming of $b(w)$ will be evident from next lemma.
Note that $b(w)$ is an integer since 
$b(w)=\frac{m}{D} \frac{D}{\gcd(D,m)} a'(w) = \frac{m}{gcd(D,m)} a'(w)$, 
where $a'(w)$ denotes the reduced stabilizer period of $w$. 

Consider the map $\phi_w\colon Gw\to\C$ defined by $\phi_w(w):= \det(g)^{a(w)}$,
which is well defined by the definition of the stabilizer period~$a(w)$.

\begin{lemma}\label{le:1-dim}
\begin{enumerate}
\item We have $\phi_w(tv)  = t^{b(w)} \phi_w(v)$ for $t\in\Ct$ and $v\in Gw$. 

\item $\Oh(Gw)^{\SL_m}_{d}$ is one-dimensional if $d$ is a multiple of the degree period $b(w)$, and zero otherwise.
More specifically, if $d=b(w)k$ with $k\in\Z$, we have $\Oh(Gw)^{\SL_m}_{d} = \C \cdot (\phi_w)^k$.

\item For $k\in\N$, 
we have $\Oh(\ol{Gw})^{\SL_m}_{b(w)k} = \IC \cdot (\phi_w)^k$ iff $(\phi_w)^{k}$ has a regular extension to $\ol{Gw}$, otherwise  $\Oh(\ol{Gw})^{\SL_m}_{b(w)k}=0$. 
\end{enumerate} 
\end{lemma}

\begin{proof}
1. For $g=t \id_m$ we have $\det(g) =t^m$ and $gw = t^D w$. Hence,  
$$
  \phi_w(t^D w) = \phi_w(gw) = \det(g)^{a(w)} = (t^m)^{a(w)} = (t^D)^{b(w)}  .
$$

2. Generally, 
let $V_G(\la)$ denote an irreducible $G$-module of highest weight~$\la$ 
and $V_G(\la)^{\stab(w)}$ denote its subspace of $\stab(w)$-invariants. 
By the algebraic Peter-Weyl theorem~\cite[\S4.1]{BLMW:11},  
$V_G(\la)^*$ occurs in $\Oh(Gw)$ with the multiplicity 
$\dim V_G(\la)^{\stab(w)}$. 
Suppose now that $V_G(\la)=\C$ corresponds to the character 
$\g\mapsto \det(g)^\ell$, where $\ell\in\Z$.
Since $\dim V_G(\la)=1$, $\dim V_G(\la)^{\stab(w)} \le 1$. 
Moreover,
$\dim V_G(\la)^{\stab(w)}$ is nonzero iff 
$\det(g)^\ell = 1 $ for all $g\in\stab(w)$. 
This means that $\ell $ is a multiple of $a(w)$, say 
$\ell =a(w) k$. Hence $F=(\phi_w)^k$. 

3.  The third assertion is immediate from the second.
\end{proof}

We note the following useful relation between the degree period and the 
stabilizer period: 
\begin{equation}\label{eq:Db=ma}
D\, b(w) = m\, a(w).
\end{equation}

We turn now to the ring $\Oh(\ol{Gw})^{\SL_m}$ of $\SL_m$-invariants 
of the  orbit closure $\ol{Gw}$ and study in which degrees~$d$ 
there is a nonzero contribution. 

\begin{definition}\label{def:E(w)} 
The {\em  degree monoid} of $w\in\Sym^D\C^m$  is defined as 
$$
 \E(w) := \big\{d\in\N \mid \Oh(\ol{Gw})^{\SL_m}_d\ne 0 \big\} . 
$$
We call the minimal positive element of $\E(w)$ the 
{\em minimal degree of $w$} and denote it by $e(w)$. 
\end{definition}

The minimal degree is well-defined due to the following result 
whose proof is postponed to Section~\ref{se:pf_group_S_w}

\begin{theorem}\label{th:group_S_w}
The degree monoid~$\E(w)$ generates the group $b(w)\Z$.  
\end{theorem}

We can give the numerical semigroup $\frac{1}{b(w)}\E(w)$ a somewhat different interpretation. 
Recall the map $\phi_w\colon Gw\to\C,\, g\mapsto \det(g)^{a(w)}$. 

\begin{definition}
The {\em exponent monoid} $\E'(w)$ of $w\in\Sym^D\C^m$ is defined as 
$$
 \E'(w) := \big\{ k\in\N \mid \mbox{$(\phi_w)^k$ has a regular extension to $\ol{Gw}$} \big\} .
$$ 
The {\em minimal exponent} $e'(w)$ is defined as the minimal positive element of $\E'(w)$. 
\end{definition}

Lemma~\ref{le:1-dim}(3) immediately implies that $\E(w) = b(w)\E'(w)$ and 
$e(w)=b(w)e'(w)$. 
We illustrate now these notions for generic forms. 

\begin{definition}\label{def:E(D,m)}
The {\em generic degree monoid of forms} of format $(D,m)$ is defined as 
$$
 \E(D,m) := \big\{d\in\N \mid \Oh(\Sym^D\C^m)^{\SL_m}_d \ne 0 \big\} .
$$
We call the minimal positive element of $\E(w)$ the 
{\em generic minimal degree}  of format $(D,m)$ and denote it by $e(D,m)$. 
\end{definition}

It is easy to see that $\E(w) = \E(D,m)$, and hence $e(w) = e(D,m)$,
for almost all $w\in \Sym^D\C^m$. 
Theorem~\ref{th:group_S_w} implies that 
$\E(D,m)$ generates the group $b(D,m)\Z$, 
where the  {\em generic degree period} $b(D,m)$ is given by  
$$
  b(D,m):=\frac{m}{D}a(D,m) = a'(D,m) \frac{m}{\gcd(D,m)} .
$$
Recall that $a'(D,m)$ is determined in Theorem~\ref{th:aDm}.

\begin{example}\label{ex:Emon}
We have $b(2,2)=2$ and $\E(2,2)=2\N$,  cf.~\cite[p.~153]{FH:91}. 
We have $b(3,3)=2$. 
Using the {\sc schur} software package (http://schur.sourceforge.net) 
we calculated that 
$\E(3,3)=2(\N\setminus\{ 1\})$
and hence $e(3,3)=2\cdot 2 =4$. 
Moreover, $b(4,4)=1$ and  
$\E(4,4)=\{0,4,6,7,8,10,11,\ldots\}$, 
so the gaps are $1,2,3,5,9$ 
and we have $e(4,4)=4$. 
\end{example}

In the sequel, we assume that $0 \neq w\in\Sym^D\C^m$ is polystable.
We can now define the key objects of study of this paper.

\begin{definition}\label{def:reg}
Let $w\in\Sym^D\C^m$ be polystable. 
By the {\em fundamental invariant}~$\Phi_w$ of $w$ 
we understand the $\SL_m$-invariant $\Phi_w \in \Oh(\ol{Gw})$
of the (minimal) degree $e(w)$ satisfying $\Phi_w(w)=1$.
\end{definition}

The explicit descriptions of the fundamental invariant~$\Phi_w$ for certain~$w$ 
(cf.\ Theorem~\ref{th:newmain} and Theorem~\ref{th:explicit-invar-formsODD})
are the main results of this paper concerning forms.  

We note that $\Phi_w = (\phi_w)^{e'(w)}$ on $Gw$ by Lemma~\ref{le:1-dim}. 
In particular, $\Phi_w$ is nonzero on the orbit $Gw$.


\begin{proposition}\label{le:ext-zero}
\begin{enumerate}
\item The zero set of $\Phi_w$ in $\ol{Gw}$ equals the boundary $\ol{Gw}\setminus Gw$. 
\item The ring $\Oh(Gw)$ of regular functions on $Gw$ equals 
the localization of $\Oh(\ol{Gw})$ with respect to fundamental invariant $\Phi_w$, i.e., 
$$
 \Oh(Gw) = \Oh(\ol{Gw})_{\Phi_w} 
   = \Big\{ \frac{f}{(\Phi_w)^s} \mid f\in\Oh(\ol{Gw}), s\in\N \Big\} .
$$
\end{enumerate}
\end{proposition}

\begin{proof}
1. It is sufficient to prove that the fundamental invariant $\Phi_w$ vanishes on 
the boundary $\ol{Gw}\setminus Gw$.
For this, let $u\in\ol{Gw}\setminus Gw$ and assume that $g_n\in G$ is such that 
$\lim_{n\to\infty} g_n w = u$ (with respect to the Euclidean topology). 
Since $\Phi_w(g_n) = (\phi_w(g_n))^{e'(w)}$, 
it is sufficient to show that 
$\lim_{n\to\infty} \det(g_n) = 0$.

Let $\|\ \|$ be a norm on $\C^{m\times m}$ and  
write $G_s:=\SL_m$. 
Since $G_s w$ is closed and $0\not\in G_s w$, 
there is $\e>0$ such that 
$\|\tilde{g} w\| \ge \e$ for all $\tilde{g} \in G_s$.
For each~$n$ there are $\tilde{g}_n\in G_s$ and $t_n\in\Ct$ such that
$g_n = t_n\, \tilde{g}_n$.
Hence $\|g_n w\| = |t_n|\, \| \tilde{g}_n w\|$.
Since $\lim_{n\to\infty} \|g_n w\| = \|u\|$ and
$\|\tilde{g}_n w\| \ge \e$ we conclude that
$|t_n | \le \|g_n w\| /\e$ is bounded.

If $\lim_{n\to\infty} t_n = 0$ were false, 
there would  be some nonzero limit point $\d$
of the sequence~$(t_n)$. After going over
to a subsequence, we have
$\lim_{n\to\infty} t_n = \d$.
From $g_nw = t_n\, \tilde{g}_n w$ we get
$\lim_{n\to\infty} \tilde{g}_n w = \d^{-1} u$.
Hence $\d^{-1}u \in \overline{G_s u} = G_s u$,
which implies the contradiction $u\in Gw$.

2. The second assertion follows from general principles; 
cf~\cite[\S4.2, p.~50]{shaf:94}.
\end{proof}

According to Lemma~\ref{le:ext-zero}, 
extending $\phi_w \colon Gw\to\C$ to the boundary $\ol{Gw}\setminus Gw$ by the value zero 
yields a continuous function $\ol{Gw}\to\C$ with respect to the 
$\C$-topology on $\ol{Gw}$. However, if $1<e'(w)$, that is, 
$b(w) < e(w)$, this extension cannot be regular on $\ol{Gw}$. 
This implies that $\ol{Gw}$ is not normal.

Here is a more detailed result, telling us that 
$\ol{Gw}$ is a complicated object if $b(w) < e(w)$. 
We postpone its proof to Section~\ref{se:pf_Phi_w-zeroset},
since it heavily relies on the arguments used for the proof of Theorem~\ref{th:group_S_w}.

\begin{theorem}\label{le:Phi_w-zeroset}
If we assume that $b(w) < e(w)$, that is, $1<e'(w)$, then 
the vanishing ideal of the boundary $\ol{Gw}\setminus Gw$ in $\ol{Gw}$
is strictly larger than the principal ideal $\Phi_w\Oh(\ol{Gw})$.
Moreover, $\ol{Gw}$ is not a normal algebraic variety. 
In particular, $\Oh(\ol{Gw})$ is not a valuation ring. 
\end{theorem}

\subsection{Constructing invariants of forms}
\label{se:constr-forms}

We will construct the fundamental invariant in several cases.
This section explains the recipe underlying these constructions.
While the construction process always returns an element $P$ in $\Oh(\Sym^D\C^m)_d^{\SL_m}$,
it does not guarantee the nonzeroness of~$P$.
Nonzeroness has to be verified separately, for example by evaluation of $P$ at a carefully chosen point.

For $\Oh(\Sym^D\C^m)_d^{\SL_m}$ to be nonzero we need that $s := Dd/m$ is an integer.
We start with an $m \times s$ array $T$ of numbers whose entries are $\{1,2,\ldots,d\}$, each number appearing exactly $D$ times
with the restriction that no number occurs in a column more than once.
This is a special case of a Young tableau.
For example, if $D=6$, $d=4$, $m=4$ we could have
\begin{equation}\label{eq:exagenericinv}
\Yvcentermath1T=\young(111111,222222,333333,444444).
\end{equation}
Tableaux of this type will be discussed in Section~\ref{se:FI-forms}.
Or, if for example $D=3$, $d=4$, $m=2$ we could have
\begin{equation}\label{eq:exa6col}
\Yvcentermath1T=\young(111222,334344),
\end{equation}
which we will treat in Section~\ref{se:mindegforms}.
An example for the case that we treat Section~\ref{subsubsec:Dodd} looks like
\begin{equation}\label{eq:exa4col}
\Yvcentermath1T=\young(1234,4123,3412),
\end{equation}
where $m=D=3$, $d=4$.

To a fixed tableau~$T$ as above we assign now a bijection
\begin{equation}\label{eqn:bij}
[D]\times [d] \to [m] \times [s],\ 
(\iota,i) \mapsto (T^{\iota,i},T_{\iota,i}) = (\kappa,j) 
\end{equation}
as follows. 
We read the tableau~$T$ columnwise from left to right. 
For $(\iota,i) \in [D]\times [d]$ 
we denote by 
$(\kappa,j)$ 
the position of $T$ with the 
$\iota$th appearance of $i$
(recall that $i$ can appear in a column at most once).
For example, in \eqref{eq:exa4col}, the second 4 from left to right appears at position $(3,2)$, so 
$(T^{2,4}, T_{2,4}) =(3,2)$. 
The inverse of the bijection~\eqref{eqn:bij} is given by 
\begin{eqnarray}\label{eqn:bij-inv}
[m]\times[s]\to[D]\times[d] \to,\ 
(\kappa,j) &\mapsto& (\tilde{T}(\kappa,j),T(\kappa,j)) = (\iota,i),
\end{eqnarray}
where $i$ is the entry in $T$ at the position $(\kappa,j)$ 
and $\iota$ equals the number of entries in $T$ in the first $j$ columns that 
contain $i$. 

A map $v\colon [m]^D \to \C$ provides a coordinate description of the tensor 
$$
 v:=\sum_{1\le \nu_1,\ldots,\nu_D \le m} v(\nu_1,\ldots,\nu_D)\  \ket{\nu_1\cdots \nu_D} 
   \in\bigotimes^D\C^m ,
$$
that we denote by the same symbol. We used here the notation 
$\ket{\nu_1\cdots \nu_D} = \ket{\nu_1} \ot\cdots\ot\ket{\nu_D}$, where 
$\ket{\nu}$ denotes the $\nu$th standard basis vector of $\C^m$.  
To a tableau~$T$ as above we assign a homogeneous degree~$d$ polynomial $P_T \colon \ot^D \C^m \to \C$
by setting:
\begin{equation}\label{eq:generalformsinvariant}
P_T(v) := \sum_{\sigma_1, \ldots, \sigma_s} \left[ \prod_{j=1}^s \sgn(\sigma_j) \right] 
   \prod_{i=1}^d v(\sigma_{T_{1,i}}(T^{1,i}),\ldots,\sigma_{T_{D,i}}(T^{D,i})  ),
\end{equation}
where the sum is over all $s$-tuples $(\sigma_1,\ldots,\sigma_s)$ of permutations of $[m]$. 

Before proving that $P_T$ is an $\SL_m$-invariant, let us look at the special case 
where $d=m$ (thus $s=D$) and 
$T$ equals the $m \times D$ tableau that has all $i$s in row $i$, see the tableau~\eqref{eq:exagenericinv}.  
In this case, the expression~\eqref{eq:generalformsinvariant}
takes the considerably simpler form 
\begin{equation}\label{eq:defP}
 P_{D,m}(v) := P_{T}(v) = \sum_{\s_1,\ldots,\s_D} \ \Big[\prod_{j=1}^D \sgn(\s_j) \Big] \cdot\prod_{i=1}^m v(\s_1(i),\ldots,\s_D(i)),
\end{equation}
where the sum is over all $D$-tuples $(\s_1,\ldots,\s_D)$ of permutations of $[m]$.
Remarkably, the invariant $P_{D,m}$ was already studied by Cayley in 1843, see~\cite{cayley-coll-work}.
We shall investigate $P_{D,m}$ in Section~\ref{se:FI-forms} in more detail, 
but already remark here that 
for $D=2$, it specializes to the usual determinant of a matrix: 
$P_{2,m}(w) = m!\, \det[w(\nu_1,\nu_2)]$, 
where $w=\sum_{\nu_1,\nu_2} w(\nu_1,\nu_2) \ket{\nu_1\nu_2}$.

\begin{theorem}\label{thm:PTnonzero}
We have $P_T(gv)=(\det g)^s P_T(v)$ for all $g \in \GL_m$ and $v \in \otimes^D\IC^m$.
\end{theorem}
\begin{proof}
We recall that $gv := g^{\ot D}(v)$ defines the action of $\GL_m$ on $\ot^D\C^m$. 
In coordinates, $w=gv$ reads as follows: 
\begin{equation}\label{eq:tens-op}
 w(\mu_1,\ldots,\mu_D) = \sum_{r:[D]\to[m]} \ v(r) \prod_{\iota=1}^D g(\mu_\iota;r(\iota)),
\end{equation}
where the sum is over all maps $r$ from $[D]$ to $[m]$.
Let $w=gv$ and $\sigma_1\ldots,\sigma_s \in S_m$. We use \eqref{eq:tens-op} to write
\[
w(\sigma_{T_{1,i}}(T^{1,i}),\ldots,\sigma_{T_{D,i}}(T^{D,i})) = \sum_{r:[D]\to[m]} v(r) \prod_{\iota=1}^D g(\sigma_{T_{\iota,i}}(T^{\iota,i});r(\iota)).
\]
Plugging this into the defining formula \eqref{eq:generalformsinvariant} for $P_T(w)$ we obtain
\begin{equation*}
\begin{split}
 P_{T}(w) &= \sum_{\s_1,\ldots,\s_s} \, \Big[\prod_{j=1}^s\sgn (\s_j) \Big] \cdot \sum_{r_1,\ldots,r_d} \ 
                       \prod_{i=1}^d \Big[ v(r_i) \, \, \prod_{\iota=1}^D \, g(\sigma_{T_{\iota,i}}(T^{\iota,i});r_i(\iota)) \Big] \\
         &= \sum_{\s_1,\ldots,\s_s} \ \sum_{r_1,\ldots,r_d} \, \Big[\prod_{j=1}^s\sgn (\s_j) \Big] \cdot \Big[\prod_{i=1}^d v(r_i)\Big] 
                        \, \prod_{i=1}^d \prod_{\iota=1}^D \, g(\sigma_{T_{\iota,i}}(T^{\iota,i});r_i(\iota)) \\
         &= \sum_{\s_1,\ldots,\s_s} \ \sum_{r_1,\ldots,r_d} \, \Big[\prod_{i=1}^d v(r_i)\Big]\ \Big[\prod_{j=1}^s\sgn (\s_j) \Big] \cdot 
                        \, \prod_{\iota=1}^D \prod_{i=1}^d \, g(\sigma_{T_{\iota,i}}(T^{\iota,i});r_i(\iota)) \\
         &= \sum_{r_1,\ldots,r_d} \ \Big[\prod_{i=1}^d v(r_i) \Big] \sum_{\s_1,\ldots,\s_s} \ \Big[\prod_{j=1}^s\sgn (\s_j) \Big] \cdot 
                        \, \prod_{\iota=1}^D \prod_{i=1}^d \, g(\sigma_{T_{\iota,i}}(T^{\iota,i});r_i(\iota)) \\
         &\stackrel{\eqref{eqn:bij-inv}}{=} \sum_{r_1,\ldots,r_d} \ \Big[\prod_{i=1}^d v(r_i) \Big] \sum_{\s_1,\ldots,\s_s} \ \Big[\prod_{j=1}^s\sgn (\s_j) \Big] \cdot
                        \, \prod_{j=1}^s \prod_{\kappa=1}^m \, g(\sigma_j(\kappa);r_{T(\kappa,j)}(\tilde T(\kappa,j))) \\ 
         &= \sum_{r_1,\ldots,r_d} \ \Big[\prod_{i=1}^d v(r_i) \Big] \sum_{\s_1,\ldots,\s_s} \ \prod_{j=1}^s\Big[\sgn (\s_j)   \prod_{\kappa=1}^m \, g(\sigma_j(\kappa);r_{T(\kappa,j)}(\tilde T(\kappa,j))) \Big]\\ 
         &= \sum_{r_1,\ldots,r_d} \ \Big[\prod_{i=1}^d v(r_i) \Big]  \underbrace{\prod_{j=1}^s\Big[\sum_{\s \in S_m}\sgn (\s)   \prod_{\kappa=1}^m \, g(\sigma(\kappa);r_{T(\kappa,j)}(\tilde T(\kappa,j))) \Big]}_{=:\Delta(r_1,\ldots,r_d)}.
\end{split}
\end{equation*}
For a tuple $r=(r_1,\ldots,r_d)$ of maps $r_i$, each mapping from $[D]$ to $[m]$, let us write $\Delta(r):=\Delta(r_1,\ldots,r_d)$.
We interpret the tuple $r$ as
a map $r : [D] \times [d] \to [m]$. Using the bijection \eqref{eqn:bij-inv} we can reinterpret $r$ as a map $r : [m] \times [s] \to [m]$,
i.e., $r$ assigns a number from $[m]$ to each position in $T$.
This assignments results in a tableau that we denote by $\Gamma(r)$.
The tableau $\Gamma(r)$ has the same shape as $T$, but its entries are taken from $[m]$.
Lemma~\ref{le:easy-cancel} below tells us that $\Delta(r)=0$ if at least one column in $\Gamma(r)$ is not a permutation and
otherwise $\Delta(r) = \det(g)^s\prod_{j=1}^s \sgn(\gamma(r,j))$, where $\gamma(r,j)\in S_m$ denotes the $j$th column of $\Gamma(r)$, 
that is, 
$\gamma(r,j)(\kappa) = r_{T(\kappa,j)}(\tilde{T}(\kappa,j))$. 
So it suffices to sum over all $s$-tuples of permutations $\gamma_1,\ldots,\gamma_s$ of $[m]$ and we get
\[
P_T(w) = \det(g)^s \sum_{\gamma_1,\ldots,\gamma_s} \Big[ \prod_{j=1}^s \sgn(\gamma_j) \Big] \cdot \prod_{i=1}^d 
   v(\gamma_{T_{1,i}}(T^{1,i}),\ldots,\gamma_{T_{D,i}}(T^{D,i})) = \det(g)^s P_T(v),
\]
which completes the proof. 
\end{proof}

\begin{lemma}\label{le:easy-cancel}
Let $g\in\GL_m$ and $\gamma\colon [m] \to [m]$ be a map. Then the sum 
$\sum_{\s\in S_m} \sgn(\s) \prod_{\kappa=1}^m g(\s(\kappa);\gamma(\kappa))$ equals 
$\sgn(\gamma) \det(g)$ if $\gamma$ is a permutation and zero otherwise.
\end{lemma}

\begin{proof}
Writing $\tilde{g}(k;\kappa) := g(k;\gamma(\kappa))$, we get 
$$
  \mbox{$\sum_{\s\in S_m} \sgn(\s) \prod_{\kappa=1}^m g(\s(\kappa);\gamma(\kappa)) 
  = \sum_{\s\in S_m} \sgn(\s) \prod_{\kappa=1}^m \tilde{g}(\s(\kappa);\kappa)  
  = \det(\tilde{g})$.} 
$$
If $\gamma$ is not a permutation, then $\tilde{g}$ has two repeated columns 
and hence $\det(\tilde{g}) =0$. Otherwise, clearly  
$\det(\tilde{g})= \sgn(\gamma)\,\det(g)$.
\end{proof}

\begin{remark}
A Young tableau $T$ is called \emph{semistandard} if the entries in each column of $T$ increase from top to bottom
and in each row they never decrease from left to right.
The example \eqref{eq:exagenericinv} is semistandard, while \eqref{eq:exa6col} and \eqref{eq:exa4col} are not.
The degree $d$ invariant space $\Oh(\Sym^D\C^m)_d^{\SL_m}$ is generated by the set of all invariants $P_T$,
where $T$ runs over all semistandard tableaux of shape $m \times s$ with exactly $D$ entries of each number $1,\ldots,d$.
The proof works similarly to the argument in Theorem~\ref{th:fund-invar-tensor}.
See for example \cite[Sec.~3(A)]{ike:12b}.
\end{remark}

\subsection{Generic fundamental invariant of forms}
\label{se:FI-forms}

The following result is due to Howe~\cite{howe:87}.

\begin{theorem}[Howe]\label{th:howe}
We have 
$\Oh(\Sym^D\C^m)^{\SL_m}_d = 0$ if $d< m$. 
Moreover, $\Oh(\Sym^D\C^m)^{\SL_m}_m$ is one-dimensional if $D$ is even and 
zero-dimensional if $D$ is odd. 
\end{theorem}

If $D$ is even, then Theorem~\ref{th:howe} states that, up to a scaling factor, 
there is exactly one homogeneous $\SL_m$-invariant 
$P_{D,m}\colon \Sym^D \C^m \to \C$  of degree~$m$
(and no nonzero invariant of smaller degree). 
We call $P_{D,m}$ the {\em generic fundamental invariant} of $\Sym^D\C^m$. 

Before discussing the explicit formula for $P_{D,m}$, we observe the following.

\begin{theorem}\label{th:newmain}
Assume that $w\in Sym^D\C^m$ is polystable. Then $m \le e(w)$.  
If $D$ is odd, then this inequality is strict. 
Moreover, if $D$ is even, then we have $m = e(w)$ iff $P_{D,m}(w) \ne 0$. 
In this case, we have 
$$
 \mbox{$\Phi_w =  P_{D,m}(w)^{-1} P_{D,m}$\quad on\ $\ol{Gw}$.} 
$$ 
\end{theorem}

\begin{proof}
Note that 
$\Oh(\ol{Gw})^{\SL_m}_d\ne 0$ implies $\Oh(\Sym^D\C^m)^{\SL_m}_d\ne 0$ 
and hence $m\le d$ due to Theorem~\ref{th:howe}. 
(If $D$ is odd we even have $m<d$.) 
If $D$ is even, then Theorem~\ref{th:howe} implies that 
$\Oh(\ol{Gw})^{\SL_m}_{e(w)}$ is generated by the restriction of $P_{D,n}$. 
So $\Oh(\ol{Gw})^{\SL_m}_{e(w)}$ does not vanish iff $P_{D,n}(w)\ne 0$,
and in this case we have 
$\Phi_w =  P_{D,m}(w)^{-1} P_{D,m}$.
\end{proof}

The following is an immediate consequence. 
(It is instructive to test this in Example~\ref{ex:Emon}.) 

\begin{corollary}
If $D$ is even, then the generic minimal degree is given by 
$e(D,m) = m$. 
\end{corollary}

We can now conclude that orbit closures of forms are rarely normal. 
In fact, a small stabilizer period already implies nonnormality of the orbit closure. 
(Recall that $a'(D,m)=1$ except in a few cases
by Theorem~\ref{th:aDm}.)

\begin{corollary}
1. If $w\in\Sym^D\C^m$ satisfies $a(w) < D$, then $\ol{Gw}$ is not normal. 

2. Suppose that $a'(D,m)=1$ and let $w\in\Sym^D\C^m$ be generic. 
Then $\ol{Gw}$ is not normal if $D$ is odd, or if $D$ is even and $\gcd(D,m)>1$.  
\end{corollary}

\begin{proof}
1. $a(w) < D$ implies $b(w) < m\le e(w)$, where 
the last inequality is due to Theorem~\ref{th:newmain}.
Now apply Theorem~\ref{le:Phi_w-zeroset}.

2. We have $a(D,m) = a'(D,m)\frac{D}{\gcd(D,m)} < D$ by assumption. 
Now apply part one to a generic~$w$.
\end{proof}

In the following we assume that $D\ge 2$ is even. 
We take $d=m\ge 1$ and let $T$ be the $m \times D$ tableau that has all $i$s in row $i$, 
cf.~\eqref{eq:exagenericinv}. 
According to Section~\ref{se:constr-forms}, 
we have a corresponding invariant $P_T$, 
that specialices to the form $P_{D,m}$ defined in \eqref{eq:defP}.

\begin{theorem}\label{th:explicit-invar-forms}
1. We have $P_{D,m}(gv) = (\det g)^D P_{D,m}(v)$
for all $g\in\GL_m$ and $v\in\ot^D \C^m$.

2. Let $m\ge 2$. Then the polynomial $P_{D,m}$ if nonzero iff $D$ is even. 
More specifically, $P_{D,m}(X_1^D+\cdots+X_m^D) = m! \ne 0$. 
\end{theorem}

\begin{proof} 
The first assertion follows by Theorem~\ref{thm:PTnonzero}.

For the second assertion, assume $m\ge 2$.
For a permutation $\s_j$ of $S_m$ we set 
$\tilde{\s}_j := \s_j \circ (12)$. 
By the defining equation~\eqref{eq:defP}, 
\begin{equation*}
\begin{split}
 P_{D,m}(v) & = \sum_{\tilde{\s}_1,\ldots,\tilde{\s}_D} \ 
    \Big[\prod_{j=1}^D \sgn(\tilde{\s}_j) \Big] \cdot\prod_{i=1}^m v(\tilde{\s}_1(i),\ldots,\tilde{\s}_D(i)) \\
 &=  \sum_{\s_1,\ldots,\s_D} \ (-1)^D \Big[\prod_{j=1}^D \sgn(\s_j) \Big] \cdot\prod_{i=1}^m v(\s_1(i),\ldots,\s_D(i)) 
 = -  P_{D,m}(v) .
\end{split}
\end{equation*}
if $D$ is odd. Therefore, $P_{D,m}(w) = 0$ for any $w\in\ot^D\C^m$. 

Suppose now that $D$ is even and consider the symmetric tensor 
$w=\sum_{i=1}^m \ket{i\cdots i} \in \ot^D\C^m$ 
corresponding to the power sum $X_1^D +\cdots X_m^D$. Thus 
$w(\mu_1,\ldots,\mu_D) $ is nonzero iff $\mu_1=\ldots =\mu_D$.
Hence in \eqref{eq:defP}, we only have contributions if $\s_1=\ldots=\s_D$ 
and therefore
$P_{D,m}(w) = \sum_{\s \in S_m} (\sgn(\s))^D = m! \ne 0$.
\end{proof}

\begin{lemma}\label{le:irred}
$P_{D,m}$ is an irreducible polynomial of degree~$m$ if $D$ is even.
\end{lemma}

\begin{proof}
This can be directly derived from the fact that,
by Theorem~\ref{th:howe}, 
$P_{D,m}$ is a nonzero $\SL_m$-invariant on $\Sym^D\C^m$ of smallest degree.
For an analog argument see for example \cite[p.~105]{Ott:09}.
\end{proof}

\begin{problem}\label{pr:compl-PDm}
What is the computational complexity to evaluate $P_{D,m}$?
\end{problem}

\subsubsection{The case where $D=m$ is odd}\label{subsubsec:Dodd}
We now treat the case where $D$ is odd, but only in the special case $D=m$.

\begin{theorem}[Howe]\label{th:howeODD} If $D$ is odd,
$\dim \Oh(\Sym^D\C^D)^{\SL_D}_{D+1} = 1$, otherwise \mbox{$\Oh(\Sym^D\C^D)^{\SL_D}_{D+1}=0$}.
\end{theorem}

The tableau $T$ that we base our construction of the generic fundamental invariant on
is the $D \times (D+1)$ tableau $T$ that has the entry $\big((j-i+1) \text{ mod } (D+1)\big)$ in row $i$ and column $j$.
Here we used the convention that $k \text{ mod } (D+1) \in \{1,\ldots,D+1\}$.
An example is provided in \eqref{eq:exa4col}.
We write $P_D:=P_T$.
Interpreting \eqref{eq:generalformsinvariant} we see that
\begin{equation}\label{eq:def:PD}
P_D := \sum_{\sigma_1,\ldots,\sigma_{D+1}}\Big[\prod_{j=1}^{D+1}\sgn(\sigma_j)\Big]\prod_{i=1}^{D+1} v(\sigma_i(1),\ldots,\sigma_{i+D-1}(D)),
\end{equation}
where the indices of $\sigma$ have to be taken modulo~$D+1$.

\begin{theorem}\label{th:explicit-invar-formsODD}
1. We have $P_D(gv) = (\det g)^{D+1} P_D(v)$
for all $g\in\GL_D$ and $v\in\ot^D \C^D$.

2. The polynomial $P_D$ if nonzero iff $D$ is odd.
\end{theorem}

\begin{proof}
The first assertion follows by Theorem~\ref{thm:PTnonzero}.
Suppose now that $D=m$ is odd and for $\alpha\in\IC$ consider the symmetric tensor 
$w=\alpha\ket{(1+2+\cdots+D)\cdots (1+2+\cdots+D)} + \sum_{i=1}^D \ket{i\cdots i} \in \ot^D\C^D$ 
corresponding to the polynomial $\alpha(X_1+X_2+\cdots+X_D)^D + X_1^D + X_2^D + \cdots + X_D^D$.
Thus 
$w(\mu_1,\ldots,\mu_D)=1+\alpha$ iff $\mu_1=\ldots =\mu_D$ and $w(\mu_1,\ldots,\mu_D)=1$ otherwise.
The evaluation of \eqref{eq:def:PD} at $w$ is a polynomial in $\alpha$ of degree at most $D+1$.
(In fact, it is homogeneous of degree~$D$, but we will not need that fact.)
We analyze the coefficient of $\alpha^{D}$ and show that it is nonzero.

Let us study the summands in the evaluation of \eqref{eq:def:PD} that contribute to the coefficient of~$\alpha^{D}$.
For such a summand there exists a cardinality $D$ subset $I \subset \{1,\ldots,D+1\}$ such that
$\sigma_i(1)=\cdots=\sigma_{i+D-1}(D)$ for all $i \in I$ (all indices of the $\sigma$ are to be taken modulo $D+1$).
Set $\bar \iota\in \{1,\ldots,D+1\}$ to be the unique number such that $(\bar\iota+1) \mod (D+1) \notin I$, so we have
$\sigma_i(1)= \cdots = \sigma_{\bar\iota}(\bar\iota-i+1) = \cdots =\sigma_{i+D-1}(D)$ for all $i \in I$.
Intuitively $\bar\iota$ is the ``rightmost'' column index of~$I$, modulo $D+1$.
We see that all $\sigma_i$, $1 \leq i \leq D+1$, are completely determined by $\sigma_{\bar\iota}$,
because in order to ensure that all $\sigma_i$ are permutations we must have
$\sigma_{\bar\iota+1}(1)=\sigma_{\bar\iota}(D)$, $\sigma_{\bar\iota+2}(2)=\sigma_{\bar\iota}(D-1)$, $\ldots$,
$\sigma_{\bar\iota+D}(D)=\sigma_{\bar\iota}(1)$.
In particular $\sigma_{\bar\iota+1}(1),\ldots,\sigma_{\bar\iota+D}(D)$ are pairwise distinct,
so $v(\sigma_{\bar\iota+1}(1),\ldots,\sigma_{\bar\iota+D}(D))=1$.
It is crucial to observe that for all $i$ the permutation $\sigma_i$ differs from $\sigma_{i+1}$ by a cyclic permutation of order $D$, i.e., by a shift by one.
Since $D$ is odd, this implies that all $D+1$ permutations $\sigma_i$ have the same sign.
Therefore the product $\prod_{i=1}^{D+1}\sgn(\sigma_i)$ is either $1^{D+1}=1$ or $(-1)^{D+1}=1$, since $D+1$ is even.
We conclude that each such summand contributes 1 to the coefficient of~$\alpha^D$,
which results in the coefficient being $D!$, which is nonzero.
\end{proof}

\begin{problem}\label{pr:whichdegree}
In which degree does the generic fundamental invariant appear for odd~$D$ and arbitrary~$m$?
\end{problem}

\subsection{Minimal degree for specific forms}
\label{se:mindegforms}

We can determine, or at least combinatorially characterize, 
the minimal degree of a few interesting forms. 

\begin{proposition}\label{pro:evenodd}
Let $w=X_1^D+\cdots X_m^D$. 
\begin{enumerate}
\item If $D$ is even, then $w$ has the degree period $b(w)=m$ 
and minimal degree $e(w) =m$.

\item If $D$ is odd, then $w$ has the degree period $b(w)=2m$. 
If additionally $2m \leq \binom{2D}{D}$, then we have $e(w) =2m$.

\item If $D$ is odd and $2m > \binom{2D}{D}$, we have $e(w) > 2m$. 

\end{enumerate}
In the cases 1 and 2, we have $e'(w)=1$. In case 3, we have $e'(w)>1$.
\end{proposition}

\begin{proof}
1. The first statement follows from Theorem~\ref{th:explicit-invar-forms}, 
Theorem~\ref{th:explicit-invar-formsODD}, and 
Proposition~\ref{pro:stab-X1-Xn-PS}.

2. 
For the construction of the invariant we choose a tableau $T$ with $m$ rows, $2D$ columns,
and entries $1,\ldots,2m$, each of which appears $D$ times.
See \eqref{eq:exa6col} for an example.
For all $i \in \{1,2,\ldots,2m\}$ the entries of $i$ all occurrences of $i$ will be located in row $\lceil \tfrac i 2\rceil$.
In particular for odd $i$ the entries of $i$ and $i+1$ are in the same row.
Let $J_i \subset \{1,\ldots,2D\}$ denote the cardinality~$D$ subsets that specifies the columns of $T$ in which $i$ appears.
The arrangement of numbers in~$T$ should be done in a way that for distinct $i$, $j$ we have $J_i \neq J_j$.
To construct such an arrangement, we proceed row by row in a greedy fashion,
i.e., we first construct $J_1$ and $J_2$, then construct $J_3$ and $J_4$, and so on.
We always choose $J_i$ such that $J_i \neq J_j$ for all $j < i$.
When we chose $J_i$ for an odd $i$, this defines $J_{i+1}$ since we assume that $J_i$ and $J_{i+1}$ are disjoint
and hence $J_{i+1}$ is the complement of $J_i$ in $\{1,2,\ldots,2D\}$.
But since $J_i \neq J_j$ for all $j < i$, it follows by construction that also
$J_{i+1} \neq J_j$ for all $j < {i+1}$.
We can proceed in this way until we run out of cardinality $D$ subsets of $\{1,2,\ldots,2D\}$,
so as long as $2m \leq \binom{2D}{D}$ we can finish this construction and end up with a tableau~$T$.

By Theorem~\ref{thm:PTnonzero}, $P_T$ is $\SL_m$-invariant.
By construction, for each $i$ and $j$ that are in different rows in $T$, we have 
\begin{equation*}\tag{$\ast$}\label{eq:ast}
J_i \cap J_j \neq \emptyset.
\end{equation*}
We now evaluate $P_T$ at the power sum $w = X_1^D + \ldots + X_m^D$.
We analyze the nonzero summands in the evaluation $P_T(w)$, see \eqref{eq:generalformsinvariant}.
Let $r_i=\lceil\frac i 2\rceil$ denote the row in which $i$ appears in $T$ and let $J_i(1),\ldots,J_i(D)$ denote the columns in which $i$ appears.
We remark that in the notation of Section~\ref{se:constr-forms} we have $T^{j,i} = r_i$ and $T_{j,i} = J_i(j)$.
Clearly $w(\mu_1,\ldots,\mu_D)=1$ iff $\mu_1=\cdots=\mu_D$ and $w(\mu_1,\ldots,\mu_D)=0$ otherwise.
Therefore we only have to consider the cases where
\begin{equation}\tag{$\ast\ast$}
\sigma_{J_i(1)}(r_i)=\sigma_{J_i(2)}(r_i)=\ldots=\sigma_{J_i(D)}(r_i).
\end{equation}
Assuming $(\ast)$ and $(\ast\ast)$ we show now that
\begin{equation}\tag{$\dagger$}
\sigma_{J_i(1)}(r_i)=\sigma_{J_i(2)}(r_i)=\ldots=\sigma_{J_i(D)}(r_i) = \sigma_{J_j(1)}(r_j)=\sigma_{J_j(2)}(r_j)=\ldots=\sigma_{J_j(D)}(r_j)
\end{equation}
in all cases where $r_i=r_j$.
First we see that among the $2Dm$ numbers $\sigma_{J_i(\kappa)}(r_i)$, $1 \leq i \leq 2m$, $1 \leq \kappa \leq D$, each number from $\{1,\ldots,m\}$ appears exactly $2D$ times:
This is clear because the $\sigma_j$ are $2D$ permutations of $[m]$.
By $(\ast\ast)$ this implies that for each number $1 \leq \kappa \leq m$ there exist exactly two distinct indices $i$ and $j$ such that $\sigma_{J_i(1)}(r_i)=\sigma_{J_j(1)}(r_j)$.
The fact that all $\sigma_j$ are permutations rules out all pairs $(i,j)$ for which $J_i \cap J_j \neq \emptyset$.
These are all pairs that are not in the same row, as stated in property $(\ast)$.
The remaining pairs are $(i,i+1)$  with $i$ odd.
This implies $(\dagger)$.

Hence $\sigma_1=\sigma_2=\ldots=\sigma_{2D}$.
Since $2D$ is even, it follows that $\prod_{j=1}^{2D}\sgn(\sigma_j)=\sgn(\sigma_1)^{2D}=1$.
This results in $m!$ summands, all of which contribute 1 to the evaluation.
Therefore $P_T(w)=m!\neq 0$.

3. The reason is a vanishing plethysm coefficient: There is no nonzero $\SL_m$-invariant in degree $2m$ if $\binom{2D}{D}< 2m$.
For the proof of the upper bound
we use the upper bound from Appendix~\ref{appendix}, Corollary~\ref{cor:plethupperbound}.
It suffices to show that there are less than $2m$ distinct cardinality $D$ subsets of $\{1,\ldots,2D\}$.
But there are only $\binom{2D}{D}$ many cardinality $D$ subsets of $\{1,\ldots,2D\}$ and $\binom{2D}{D}<2m$.
\end{proof}

A {\em Latin square of size~$n$} is map $T\colon [n]^2 \to [n]$,
viewed as a $n\times n$ matrix with entries in~$[n]$, 
such that in each row and each column each entry in~$[n]$ 
appears exactly once.
So in each column and row we get a permutation of $[n]$. 
The column sign of $T$ is defined as the product of the signs of 
column permutations. The Latin square~$T$ is called 
{\em column-even} if this sign equals one, otherwise $T$ 
is called {\em column-odd}. 

The Alon-Tarsi conjecture~\cite{AT:92} states that the 
number of column-even Latin squares of size~$n$ is different from 
the number of column-odd Latin squares of size~$n$ if $n$ is odd. 
This conjecture is known to be true if $n=p\pm 1$ where $p$ is a prime, 
cf.~\cite{Dri:98,Gly:10}.
The following observation is due to Kumar~\cite{Kum:15} 
and Kumar and Landsberg~\cite{kumar-landsberg:15}.

\begin{proposition}\label{pro:ps-pr-expr}
Let $m$ be even. Then $e(X_1\ldots X_m) \ge m$ and equality holds 
iff the Alon-Tarsi conjecture is true for $m$.  
\end{proposition}

\begin{proof}
By Theorem~\ref{th:newmain} it is sufficient to show that 
$P_{m,m}$ evaluated at $X_1\cdots X_m$ is nonzero iff 
the Alon-Tarsi conjecture holds for~$m$. 

In order to show this, consider the symmetric tensor 
$$
 w_m := \frac{1}{m!} \sum_{\pi\in S_m} \ket{\pi(1)\cdots\pi(m)}
$$
corresponding to the polynomial $X_1\cdots X_m$. 
Hence $w(\mu_1,\ldots,\mu_m)$ is nonzero iff $\mu$ is a permutation of $[m]$.
We evaluate now $P_{m,m}(w_m)$ according to \eqref{eq:defP}. 
If we interpret $\s_1,\ldots,\s_m$ as the columns of the matrix 
$T\colon [m] \times [m] \to [m],\, (i,j) \mapsto \sigma_i(j)$, 
then we get a nonzero contribution iff $T$ is a latin square of size~$m$, that is, 
in each row and in each column, each symbol in $[m]$ appears exactly once. 
Therefore, $P_{m,m}(w_m)$ equals the difference of the column-even and the 
column-odd latin squares of size~$m$. 
The Alon-Tarsi conjecture~\cite{AT:92} states that this difference is nonzero 
if $m$ is even. 
\end{proof}

\begin{remark}
Proposition~\ref{pro:ps-pr-expr} can be generalized to odd $m$ by
reinterpreting the Alon-Tarsi conjecture as follows.
An $m \times d$ \emph{Latin Annulus} is defined to be an $m \times d$ matrix
where in each column and in each diagonal $\big((1,i),(2,i+1),(3,i+2),\ldots,(m,i+m-1)\big)$
each number from $1,\ldots,m$ occurs exactly once,
where the column indices are taken modulo $d$.
The column-sign of a Latin Annulus is the product of the signs of the permutations in its columns.
The Alon-Tarsi conjecture is equivalent to saying that for even $m$ the number of even $m \times m$ Latin Annuli
is different from the number of odd $m \times m$ Latin Annuli.
Analogously to the proof of Proposition~\ref{pro:ps-pr-expr} we see that
for odd $m$ $P_{m}(w_m)\neq 0$ iff the number of
even $m \times (m+1)$ Latin Annuli is different from the number of odd $m \times (m+1)$ Latin Annuli.
Therefore for odd $m$ we have $e(X_1\ldots X_m) = m+1$ iff
the number of
even $m \times (m+1)$ Latin Annuli is different from the number of odd $m \times (m+1)$ Latin Annuli.
We verified that the number of even $m \times (m+1)$ Latin Annuli differs from the number of odd $m \times (m+1)$ Latin Annuli for $m=1$, $3$, $5$, and $7$.
\end{remark}

The period of $\det_n$ and $\per_n$ was characterized in  Theorem~\ref{th:stab-det}.
Combined with Theorem~\ref{th:newmain}, we obtain that 
$e(\det_n) \ge n^2$, with equality holding iff $P_{n,n^2}(\det_n) \ne 0$. 
(An analogous statement holds for $\per_n$.)

\begin{remark}
$\det_2$ and $\per_2$ are quadratic forms of full rank, they are equivalent to 
$X_1^2+\cdots+X_4^2$. Hence 
$P_{2,4}(\det_2)$ and $P_{2,4}(\per_2)$ are nonzero due to 
Theorem~\ref{th:explicit-invar-forms}.
\end{remark}

For geometric complexity, it would be of great interest to know whether $e(\det_n) = n^2$.
If equality holds, then we have an explicit description of the 
fundamental invariant of $\det_n$, which provides a link
between the coordinate ring of $\GL_{n^2}\det_n$ and its closure.

We now provide a combinatorial description 
of $P_{n,n^2}(\det_n)$ in the spirit of the Alon-Tarsi conjecture 
if $n$ is even.

We consider maps 
$S\colon [n^2] \times [n] \to [n]$, viewed as matrices of format $n^2\times n$,
with entries in $[n]$, in which each row $S(i,\cdot)$ is a permutation of $[n]$.
Let us call the product of the signs of the rows the row-sign of $S$. 
We call a pair $(S,T)$ of maps $[n^2] \times [n] \to [n]$ 
an {\em admissible $n$-table} iff each column 
of the table $(S,T)$ enumerates $[n]\times [n]$, that is, for all~$j$, 
$[n^2] \to [n] \times [n], i \mapsto (S(i,j),T(i,j))$ 
is bijective.  
The sign of this permutation is well-defined after fixing an ordering of $[n]\times [n]$. 
We denote by $\csgn(S,T)$ the product of these signs and call it the column-sign of 
the admissible table $(S,T)$. The table is called  {\em column-even} iff $\csgn(S,T)=1$
and {\em column-odd} otherwise.

We define the row-sign 
of an admissible table $(S,T)$ 
as the product of the row-sign of $S$ and the row-sign of $T$.
The sign of $(S,T)$ is defined as the product of its row and column sign.
Accordingly, we call an admissible table {\em even} iff it has sign~$1$ and odd otherwise.

\begin{proposition}\label{re:eval-P-det-per}
$P_{n,n^2}(\det_n)$ is the difference between the number of even and odd 
admissible $n$-tables. 
$P_{n,n^2}(\per_n)$ is the difference between the number of column-even 
and column odd admissible $n$-tables. 
\end{proposition}

\begin{proof}
The permanent 
$\per_n = \sum_{\pi\in S_n} X_{\pi(1),1}\cdots X_{\pi(n),n}$ 
corresponds to the following symmetric tensor in
$\Sym^n \C^{n\times n}$ (denoted by the same symbol)
$$
\per_n := \frac{1}{n!} \sum_{\pi\in S_n} \sum_{\s_\in S_n} 
   \ket{ (\s\pi(1),\s(1)) \cdots (\s\pi(n),\s(n)) }
   = \frac{1}{n!} \sum_{\rho\in S_n} \sum_{\s_\in S_n} 
    \ket{ (\rho(1),\s(1)) \cdots (\rho(n),\s(n)) } .
$$
Thus an $n$-tupel $((\mu_1,\nu_1),\ldots,(\mu_n,\nu_n))$ in $[n]^2$ 
lies in the support of $\per_n$ iff $\mu$ and $\nu$ are permutations of $[n]$. 
Similarly, 
the determinant $\det_n = \sum_{\pi\in S_n} \sgn(\pi) X_{\pi(1),1}\cdots X_{\pi(n),n}$ 
corresponds to  
$$
\det_n := \frac{1}{n!} \sum_{\rho\in S_n} \sum_{\s_\in S_n} 
    \sgn(\rho)\sgn(\s)\cdot \ket{ (\rho(1),\s(1)) \cdots (\rho(n),\s(n)) } .
$$
The assertion now follows from 
the explicit description~\eqref{eq:defP} of fundamental invariants. 
\end{proof}

Based on Proposition~\ref{re:eval-P-det-per} 
we have verified that 
$P_{4,16}(\per_4)$ and $P_{4,16}(\det_4)$ are nonzero,
using computer calculations.

\begin{corollary}\label{cor:not-normal}
1. The orbit closure of $X_1\cdots X_m$ is not normal if $m>2$.

2. The orbit closures of $\det_n$ and $\per_n$ are not normal if $n>2$. 

\end{corollary}

\begin{proof}
Combine Theorem~\ref{le:Phi_w-zeroset} with Theorem~\ref{th:newmain} 
and the determination of the periods in 
Proposition~\ref{pro:stab-X1-Xn-PS}, 
Theorem~\ref{th:stab-det}, 
and Theorem~\ref{th:stab-per}. 
\end{proof}

We note that the orbit closures of $X_1X_2$ equals $\Sym^2\C^2$ and hence is normal.
Similarly, the orbit closures of $\det_2$ and $\per_2$ equal $\Sym^2\C^4$ and hence are normal.

Let us point out again that the non-normality of orbit closures 
is the main source of difficulty in implementing 
Mulmuley and Sohoni's strategy, see \cite{Kum:15,BHI:15}.

\section{Stabilizer period and polystability of tensors}
\label{se:prelim-T}

Our investigations here are analogous to Section~\ref{se:stab-period} 
and Section~\ref{se:polystab} for forms. 
We confine ourselves to cubic tensors. 
Recall that the group $\GL_m^3$ acts on $W:=\ot^3\C^m$ via the tensor product. 
Some of the results appeared in \cite{BI:10}. 

\subsection{Stabilizer and stabilizer period of tensors}
\label{se:stab-period-tensor}

The subgroup 
$K := \{(\zeta_1 \id_m,\zeta_2\id_m,\zeta_3\id_m) \mid \zeta_i\in\Ct, \zeta_1\zeta_2\zeta_3=1 \}$ 
acts trivially on $W=\ot^3\C^m$, so we have an induced action of 
$G:=\GL_m^3/K$ on $W$. 
The {\em stabilizer} $\stab(w)$ of a tensor $w\in\ot^3\C^m$ is 
defined as the following closed subgroup of~$\GL_m^3$: 
\begin{equation}\label{eq:def-stab-tensor}
 \stab(w) := \big\{ (g_1,g_2,g_3) \in \GL^3_m\mid (g_1\ot g_2\ot g_3) w = w \big\} .
\end{equation}
Clearly, $K\subseteq \stab(w)$.  
All the information about $\stab(w)$ is contained in 
the {\em reduced stabilizer} $\stab'(w):=\stab(w)/K$ of $w$, which is defined  
as the image of $\stab(w)$ under the canonical morphism $\GL_m^3\to G$. 
We say that $w$ has a {\em trivial stabilizer} if $\stab'(w)$ consists of the unit element only.  

The group homomorphism 
\begin{equation}\label{eq:def-chi}
\chi\colon \GL_m^3\to \Ct, (g_1,g_2,g_3) \mapsto \det(g_1)\det(g_2)\det(g_3)
\end{equation}
factors through $K$; we denote the resulting homomorphism $G\to\Ct$ by $\chi$ as well. 
The image $H:=\chi(\stab(w))$ of the stabilizer under $\chi$ 
is a closed subgroup of $\Ct$ 
(cf.~\cite[\S7.4, Prop.~B]{hump:75}). 
Either $H=\Ct$ or $H$ equals the group $\mu_a$ of 
of $a$th roots of unity, where $a=|H|$. 

\begin{definition}\label{def:period-tensor}
The {\em stabilizer period} $a(w)$ of a tensor $w\in\ot^3\C^m$ is defined as the 
order of $\chi(\stab(w))$.
\end{definition}

It is clear that a tensor with trivial stabilizer has the stabilizer period~$1$. 
It follows from general principles that 
there exists $a(m)$ such that almost all $w\in\ot^3\C^m$ have the 
stabilizer period $a(m)$; 
see~\cite{richardson:72} or Popov and Vinberg in~\cite[\S7]{AG4}. 

\begin{theorem}\label{th:stab_gen_tensor}
We have $a(m)=1$ for $m\ge 3$ and $a(2)=2$.
\end{theorem}

\begin{proof}
A.M.~Popov~\cite[Thm.~2]{popov-AM:87} proved that if $m>3$, 
then almost all $w\in\ot^3\C^m$ have a trivial stabilizer group, hence $a(m)=1$. 
(His result is more general and covers also noncubic dimension formats.) 

The isomorphism classes of tensors in $\ot^3\C^3$ have been 
classified by Thrall and Chanler~\cite{thrall-chanler:38}, see also \cite{ng:95}.
This reveals that $a(w)=1$ for a generic  $w\in\ot^3\C^3$. 

Almost all $w\in\ot^3\C^2$ are in the same $\GL^3_2$-orbit as the unit tensor $\lan 2\ran$, 
which has a nontrivial stabilizer and stabilizer period two, cf.\ Theorem~\ref{cor:stabut}. 
\end{proof}

We shall next determine the stabilizer and the stabilizer period of some interesting tensors.
We first proceed with a general useful reasoning
(compare~\cite[Chap.~2]{chen-kayal-wigderson:10})
that will also be used in Appendix~\ref{appendix}. 
The {\em Hessian} $H(f)$ of a polynomial $f$ in the variables $X_1,\ldots,X_n$
is defined as the symmetric matrix given by the second order 
partial derivatives of $f$: 
$$
 H(f) := \big[\partial^2_{X_i X_j} f\big]_{i,j\le n} . 
$$
Let $A\in\GL_n$ and put $\tilde{f}(X):= f(AX)$, 
where $X=(X_1,\ldots,X_n)^T$. 
The chain rule easily implies that 
$$
 H(\tilde{f})(X) = A^T\cdot H(f)(AX) \cdot A ,
$$ 
which implies 
$\det H(\tilde{f})(X) = (\det A)^2 \cdot\det H(f)(AX)$. 
In particular, if $f(AX)=f(X)$, we get 
\begin{equation}\label{eq:detH-stab}
 \det H(f)(X) = (\det A)^2 \cdot\det H(f)(AX) .
\end{equation}

It is useful to interpret a tensor $w\in\ot^3 \C^m$ as 
a trilinear form $\ot^3 (\C^m)^* \to \C$. 
If $X_1,\ldots,X_m$, $Y_1,\ldots,Y_m$, and $Z_1,\ldots,Z_m$ 
denote the standard basis of $(\C^m)^*$ in the three factors, respectively, 
we can write $w$ as the multilinear polynomial 
$w=\sum_{ijk} w_{ijk} X_iY_jZ_k$. 

The $m$th {\em unit tensor} is defined as 
$\lan m\ran := \sum_{i=1}^m \ket{i} \ot \ket{i} \ot\ket{i}$.
It corresponds to the trilinear form $\sum_{i=1}^m X_iY_iZ_i$. 

The following result is known, but our proof is more elementary 
than the one given in \cite{deGroote:78}; see also \cite{BI:10}.

\begin{theorem}\label{cor:stabut}
The stabilizer $H_m$ of the unit tensor $\lan m\ran$ in $\GL_m^3$ 
 is the semidirect product of the subgroup
$$
 \Tc_m:=\{(\diag(a),\diag(b),\diag(c)) \mid \forall i\ a_i b_i c_i=1\},
$$
and the symmetric group $S_m$ diagonally embedded in $G_m$ via
$\pi\mapsto (P_\pi,P_\pi,P_\pi)$. 
(Here $P_\pi$ denotes the permutation matrix of $\pi$.)
The stabilizer period of $\lan m\ran$ equals~$2$ if $m>1$.
\end{theorem}

\begin{proof}

A straightforward calculation shows that 
the Hessian $H(f)$ of the trilinear form $f=\sum_{i=1}^m X_iY_iZ_i$
is the block direct sum of the matrices 
$\left[\begin{smallmatrix} 
0   & Z_i & Y_i \\
Z_i & 0   & X_i \\
Y_i & X_i & 0
\end{smallmatrix} 
\right]$
for $i=1,\ldots,m$. 
We have 
\begin{equation}\label{eq:detHf}
\mbox{$\det H(f) = \prod_{i=1}^m X_i Y_i Z_i .$}
\end{equation}
Suppose the linear transformation $X_i\mapsto X_i',\, Y_i\mapsto Y_i',\, Z_i\mapsto Z_i'$
fixes $f$. By~\eqref{eq:detH-stab}, the determinant of the Hessian of $f$ is preserved 
up to a scalar. Hence \eqref{eq:detHf} implies that 
$\prod_{i=1}^m X_iY_iZ_i = c \prod_{i=1}^m X_i'Y_i'Z_i'$ for some $c\in\C^\times$.  
The uniqueness of factorization of polynomials implies that 
$X_i' = a_i X_{\rho(i)}$,  $Y_i' = b_i Y_{\sigma(i)}$, $Z_i' = c_i Z_{\tau(i)}$
for $\rho,\sigma,\tau\in S_m$ and $a_i,b_i,c_i\in\C^\times$. 
Since the transformation fixes $f$, we must have 
$\rho=\sigma=\tau$ and $a_ib_ic_i=1$ for all~$i$. 
\end{proof}

It is remarkable that $\lan m \ran$ is uniquely determined by its stabilizer up to a scalar.

\begin{proposition}\label{le:ut-unique}
If the stabilizer of some tensor $w\in\C^m\ot\C^m\ot\C^m$ contains 
the stabilizer $H_m$ of the unit tensor $\lan m\ran$, 
then $w=c\,\lan m\ran$ for some $c\in \C$.
\end{proposition}

\begin{proof}
Assume the stabilizer of
$w=\sum w_{ijk} \ket{ijk}$ contains $H_m$.
By contradiction, we suppose that $w_{ijk}\ne 0$ for some
$i,j,k$ with $i\ne k$.
For any $(\diag(a),\diag(b),\diag(c))\in \Tc_m$ we have
$a_i b_j c_k w_{ijk} = w_{ijk}$ and hence
$a_i b_j c_k =1 =a_k b_k c_k$, which implies
$a_i = a_k b_k/b_j$.
However, defining
$\tilde{a}_i = 2 a_i$, $\tilde{c}_i = \frac12 c_i$,
$\tilde{a}_\ell = a_\ell,\ \tilde{c}_\ell = c_\ell$ 
for $\ell\ne i$ we get 
$(\diag(\tilde{a}),\diag(b),\diag(\tilde{c}))\in \Tc_m$.
This yields the contradiction
$\tilde{a}_i = \tilde{a}_k b_k/b_j = a_k b_k/b_j = a_i$. 
We have thus shown that $w_{ijk}\ne 0$ implies $i=k$.
By symmetry, we conclude that $w_{ijk}=0$ unless $i=j=k$.
Finally,
from the invariance of $w$ under $S_m$, we get $w_{iii} = w_{111}$ for all~$i$.
Hence $w=w_{111}\lan m\ran$.
\end{proof}

We turn now to the matrix multiplication tensor $\lan n,n,n\ran$, 
which can be interpreted as the trilinear form $\tr(XYZ)$, where 
$X,Y,Z$ are $n\times n$ matrices of variables. 

The following result is due to de Groote~\cite{deGroote:78} and 
relies on the Skolem-Noether theorem; see also \cite[Prop.~5.1]{BI:10}.

\begin{theorem}\label{th:stab-mamu}
The stabilizer of $\tr(XYZ)$ in $\GL_{n^2}^3$ is given by 
the transformations 
$X\mapsto AXB^{-1}, Y\mapsto BXC^{-1}, Z\mapsto CXA^{-1}$, 
where $A,B,C\in\GL_n$.
\end{theorem}

\begin{corollary}\label{cor:period-mamu}
The tensor $\lan n,n,n\ran$ of matrix multiplication has the stabilizer period~$1$.
\end{corollary}

\begin{proof}
We recall that the linear transformation 
$g_{A,B}\colon\C^{n\times n} \to \C^{n\times n},\, X \mapsto A X B^{-1}$, 
defined for $A,B\in\GL_n$, satisfies $\det (g_{A,B}) = \det(A)^n \det(B)^{-n}$. 
If $g\in\GL_{n^2}^3$ is given by the  transformations 
$X\mapsto AXB^{-1}, Y\mapsto BXC^{-1}, Z\mapsto CXA^{-1}$, 
then 
$$
\chi(g) = \det (g_{A,B}) \det (g_{B,C}) \det (g_{C,A}) = 
 \det(A)^n \det(B)^{-n} \det(B)^n \det(C)^{-n}\det(C)^n \det(A)^{-n} = 1 .
$$
Hence the period equals~$1$. 
\end{proof}

It is remarkable that $\lan n,n,n\ran$ is uniquely determined by its stabilizer up to a scalar. 

\begin{proposition}\label{pro:charstabmamu}
If the stabilizer of some tensor $w\in\ot^3\C^{n\times n}$ 
contains the stabilizer of $\lan n,n,n\ran$, 
then $w=c\,\lan n,n,n\ran$ for some $c\in \C$. 
\end{proposition}

\begin{proof}
It is helpful have a coordinate-free description of $\lan n,n,n\ran$. 
Let $U$ be a $\C$-vector space of dimension~$n$. 
Recall the canonical isomorphism 
$U^*\ot U \simeq \End(U),\, \ell \ot v \mapsto (u\mapsto \ell(u) v)$, 
which is  $\GL(U)\ti\GL(U)$-equivariant. Choosing a basis of $U$ and 
its dual basis of $U^*$, identifying $\End(U)$ with $\C^{m\times m}$, 
and $\GL(U)$ with $\GL_m$, the resulting action of $\GL_m\ti\GL_m$ 
on $\C^{m\ti m}$ turns out to be $(A,B)X = AX (B^T)^{-1}$.
(The transpose comes from going over to the dual; this is is sometimes 
confusing.) 
We note that  the space of $\GL(U)\ti\GL(U)$-invariants in $\End(U)=U^*\ot U$ 
is one-dimensional and equals $\C\Id_U$. 

Let now $U_1,U_2,U_3$ be three $\C$-vector space of dimension~$n$. 
We have a canonical isomorphism 
\begin{equation}\label{eq:iso-mamu-spaces}
 \End(U_1) \ot \End(U_2) \ot \End(U_3) \simeq 
 (U_1^* \ot U_2) \ot (U_2^* \ot U_3) \ot (U_3^* \ot U_1) ,
\end{equation}
which results from permuting the factors in a cyclic way and rearranging
parentheses. It is easy to check that the image of 
$\C\Id_{U_1} \ot \C\Id_{U_2} \ot \C\Id_{U_3}$ 
on the right-hand side corresponds to the composition 
$\Hom(U_2,U_1)\ti\Hom(U_3,U_2)\to\Hom(U_3,U_1),\, (\varphi,\psi)\mapsto \varphi\circ\psi$ 
of linear maps. Therefore, the image of 
$\Id_{U_1} \ot \Id_{U_2} \ot \Id_{U_3}$ 
is the coordinate-free description of the matrix multiplication tensor $\lan n,n,n\ran$. 

According to Theorem~\ref{th:stab-mamu}, 
the action of the stabilizer of $\lan n,n,n\ran$, realized on the 
left-hand side of \eqref{eq:iso-mamu-spaces}, is the usual 
action of $\GL(U_1)\ti\GL(U_2)\ti\GL(U_3)$.
On the other hand, the space of $\GL(U_1)\ti\GL(U_2)\ti\GL(U_3)$-invariants of 
$\End(U_1) \ot \End(U_2) \ot \End(U_3)$ is one-dimensional and 
given by  $\C\Id_{U_1} \ot \Id_{U_2} \ot \Id_{U_3}$. 
\end{proof}

\subsection{Polystability of tensors}

We call a tensor $w\in\ot^3\C^m$ {\em polystable} 
iff the $\SL_m^3$-orbit of $w$ is closed. 
Our goal is a criterion for polystability, 
analogous to Proposition~\ref{pro:stab-crit-forms}. 
We define the {\em support} of a tensor $w\in\ot^3\C^m$ as 
\begin{equation}\label{eq:def-support-tensor}
 \supp(w) := \big\{ (i,j,k)\in [m]^3 \mid w_{ijk} \ne 0 \big\} ,
\end{equation}
where $w=\sum_{ijk} w_{ijk} \ket{ijk}$. 
Suppose we have a map $\a\colon \supp(w)\to \R$. Its {\em first marginal} 
is defined by 
$\a^{1}\colon [m]\to\R,\, i\mapsto \sum_{j,k} \a(i,j,k)$, 
and its second and third marginals $\a^{2},\a^{3}$ are defined similarly. 
Note that if $\a$ is a probability distribution on $\supp(w)$, then its 
marginals are probability distributions on $[m]$. 

Let $T_m$ denote the subgroup of $\GL_m$ consisting of the diagonal matrices 
and consider the subgroup $T_m^3:=T_m\ti T_m\ti T_m$ of $\GL_m^3$.

\begin{proposition}\label{pro:stab-crit-tensors}
Let the tensor $w\in \ot^3\C^m$ satisfy the following two properties:
\begin{enumerate}
\item There is a reductive subgroup $R$ of $\SL_m^3\cap \stab(w)$ 
such that the centralizer of $R$ in $\SL_m^3$ is contained in $T_m^3$. 

\item There is a probability distribution $\a$ on $\supp(w)$ such that its 
marginals $\a^{1}, \a^{2},\a^{3}$ are the uniform distributions on $[m]$.

\end{enumerate}
Then $w$ is polystable.
\end{proposition}

\begin{proof}
We proceed as for Proposition~\ref{pro:stab-crit-forms}. 
Suppose that $w$ is not polystable and let $Y$ be a nonempty closed $\SL_m^3$-orbit in 
$\ol{\SL_m^3 w}\setminus \SL_m^3 w$. 
By the Hilbert-Mumford criterion, refined by 
Luna~\cite[Cor.~1]{luna:75} and Kempf~\cite[Cor.~4.5]{kempf:78}, 
there exists a one-parameter subgroup
$\s\colon\C^\ti\to \SL_m^3$ with $y:=\lim_{t\to 0}\s(t)w \in Y$, 
such that the image of $\s$ is contained 
in the centralizer of $R$ in $\SL_m^3$. 
By our first assumption, $\s$ maps $\Ct$ to $T_m^3$. 
So there exist $\mu,\nu,\pi\in\Z^m$ such that 
$\s(t) = (\diag(t^{\mu_1},\ldots,t^{\mu_m}), \diag(t^{\nu_1},\ldots,t^{\nu_m}),\diag(t^{\pi_1},\ldots,t^{\pi_m}))$ 
for $t\in\Ct$. 
Note that $\sum_i \mu_i = \sum_i \nu_i = \sum_i \pi_i = 0$ since the image of $\s$ is in $\SL_m^3$. 

We have
$$
 \s(t) w = \sum_{i,j,k} t^{\mu_i+\nu_j+\pi_k} \, w_{ijk} .
$$
The existence of $\lim_{t\to 0} \s(t) w$ implies that 
$\mu_i+\nu_j+\pi_k \ge 0$ for all $(i,j,k)\in\supp(w)$. 
Taking convex combinations, this condition implies that 
$\lan \a^{1},\mu\ran + \lan \a^{2},\nu\ran +\lan \a^{3},\pi\ran \ge 0$ 
for all probability distributions $\a$ on $\supp(w)$. 

Our second assumption states that there exists a probability distribution $\a$ on $\supp(w)$
whose marginals are the uniform distributions on $[m]$. 
But then $\lan \a^{1},\mu\ran = \frac{1}{m}\sum_i \mu_i = 0$ and similarly, 
$\lan \a^{2},\nu\ran = \lan \a^{3},\pi\ran = 0$. 
This implies $\mu=\nu=\pi=0$ and hence $\s(t)=1$ for all $t$. 
Therefore $y=w$, which contradicts the fact the $y$ lies in the 
boundary of the $\SL_m$-orbit of~$w$.
\end{proof}

The following result was shown by Meyer~\cite{meyer:06}. 

\begin{corollary}\label{cor:stable-forms-tensor}
The unit tensors $\lan m \ran$ and the matrix multiplication tensors $\lan n,n,n\ran$ 
are polystable.
\end{corollary}

\begin{proof}
We use Proposition~\ref{pro:stab-crit-tensors}. 
For $\lan m \ran$ we take $R:=\Tc_m$ as defined in Corollary~\ref{cor:stabut}. 
Note that $\Tc_m$ is contained in $\SL_m^3\cap\stab(\lan m\ran)$. It is easy to check 
that the centralizer of $\Tc_m$ in $\SL_m^3$ is contained in $T_m^3$. 
For the second condition, we note that the marginals of the uniform distribution 
on the support of $\lan m\ran$ are the uniform distributions on $[m]$. 
Hence $\lan m \ran$ is polystable. 

For $\lan n,n,n\ran$ we take for $R$ the group consisting of the 
the transformations 
$X\mapsto AXB^{-1}, Y\mapsto BXC^{-1}, Z\mapsto CXA^{-1}$, 
defined by diagonal matrices $A,B,C\in T_n$.
In other words, $R$ consists of the triples of diagonal matrices 
$(\diag(a_i b_j^{-1}), \diag(b_i c_j^{-1}), \diag(c_i a_j^{-1}))$, 
where $a_1,\ldots,c_n\in\C^\ti$. 
Note that $R$ is indeed contained in $\SL_{n^2}^3\cap\stab(\lan n,n,n\ran)$. 
If $g\in\SL_{n^2}$ commutes with all $\diag(a_i b_j^{-1})$, then 
$g$ must be diagonal, since we may choose the $a_i,b_j$ such that 
$a_ib_j^{-1}$ are pairwise distinct for $i,j=1,\ldots,n$.  
Hence the centralizer of $R$ is contained in $T_m^3$. 
Finally, it is easy to check that the marginals of the uniform distribution on the 
support of $\lan n,n,n\ran$ are uniform as well.
\end{proof}

\begin{proposition}\label{pro:gen-tensor-stable}
Almost all $w\in\ot^3\C^m$ are polystable. 
\end{proposition}

\begin{proof}
Recall that the $\GL_m^3$-action induces an action of  
$G=\GL_m^3/K$ on $\ot^3\C^m$. 
Theorem~\ref{th:stab_gen_tensor} states that the reduced stabilizer 
$\stab'(w)=\stab(w)/K$ is finite for almost all $w\in\ot^3\C^m$. 
The group $G_s :=\SL_m^3/K$ is a homomorphic image of $\SL_m^3$ and hence semisimple. 
Clearly, $\SL_m^3 w = G_sw$.
According to ~\cite{luna:73}, or \cite[II 4.3D, Folgerung, p.~142]{kraf:84}, 
$G_s w $ is closed for almost all~$w$ since $\stab'(w) =\stab(w)/K$ 
is finite for almost all~$w$.
\end{proof}

Any polystable tensor has a finite stabilizer period
(the proof is as for Proposition~\ref{pro:polystab-period}).

\section{Fundamental invariant of tensors}
\label{se:FI-tensors} 

We proceed similarly as in Section~\ref{se:exp-monoid-forms}. 
Throughout this section, we assume that $w\in\ot^3\C^m$ is polystable.

Consider the map $\phi_w\colon Gw\to\C$ defined by $\phi_w(w):= \chi(g)^{a(w)}$,
which is well defined by the definition of the stabilizer period~$a(w)$.
The following is analogous to Lemma~\ref{le:1-dim}.

\begin{lemma}\label{le:1-dim-T}
\begin{enumerate}
\item We have $\phi_w(tv)  = t^{m a(w)} \phi_w(v)$ for $t\in\Ct$ and $v\in Gw$. 

\item $\Oh(Gw)^{\SL_m}_{d}$ is one-dimensional if $d$ is a multiple of $m a(w)$, and zero otherwise.
More specifically, if $d=m a(w)k$ with $k\in\Z$, then we have $\Oh(Gw)^{\SL_m}_{d} = \C \cdot (\phi_w)^k$.

\item For $k\in\N$, 
we have $\Oh(\ol{Gw})^{\SL_m}_{ma(w)k} \ne 0$ iff $(\phi_w)^{k}$ has a regular extension to $\ol{Gw}$. 
\end{enumerate} 
\end{lemma}

\begin{proof}
Let $\iota(t):=(t\id_m,\id_m,\id_m)\bmod K \in G$ for $t\in\Ct$. 
We have 
 $\chi(\iota(t))=t^m$ and $\iota(t)w=tw$. We thus get 
$\phi_w(tw)=\phi_w(\iota(t)w)) = \chi(\iota(t))^{a(w)} = t^{m a(w)}$, 
hence $ma(w)$ equals the degree of $F$. 
This shows the first assertion. 
The remaining assertions are shown as for Lemma~\ref{le:1-dim}.
\end{proof}

\begin{definition}
The {\em degree monoid} $\E(w)$ of $w\in\ot^3\C^m$ is defined as 
$$
 \E(w) := \big\{d\in\N \mid \Oh(\ol{Gw})^{\SL^3_m}_d\ne 0 \big\} , 
$$
and we call the minimal positive element $e(w)$ of $\E(w)$ the 
{\em minimal degree}  of $w$. 
\end{definition}

The following is analogous to Theorem~\ref{th:group_S_w}, 
cf.\  Section~\ref{se:Editto} for the proof. 

\begin{theorem}\label{th:group_S_w-T} 
The degree monoid $\E(w)$ generates the group $m a(w) \Z$. 
\end{theorem}

Again we have $\E(w) = ma(w)\E'(w)$ with the  {\em exponent monoid} characterized by 
$$
 \E'(w) := \big\{ k\in\N \mid \mbox{$(\phi_w)^k$ has a regular extension to $\ol{Gw}$} \big\} .
$$ 
The {\em minimal exponent} $e'(w)$ of $\E'(w)$ is defined as the minimal positive element of $\E'(w)$. 

To illustrate the notions introduced, recall that 
$a(m)$ denotes the stabilizer period of a generic tensor $w\in\ot^3\C^m$.
From Theorem~\ref{th:stab_gen_tensor} we know 
that $a(m)=1$ if $m>2$ and $a(2)=2$. 
We define the {\em generic degree monoid of tensors} (of the cubic format $m$) as 
\begin{equation}\label{eq:E-T}
 \E(m) :=  \big\{d\in\N \mid \Oh(\ot^3\C^m)^{\SL^3_m}_d \ne 0 \big\} ,
\end{equation}
and we call its minimal positive value $e(m)$ the {\em generic minimal degree}.
It is easy to see that $\E(w) = \E(m)$, and hence $e(w) = e(m)$, 
for almost all $w\in \ot^3\C^m$. 
Theorem~\ref{th:group_S_w-T} implies that 
$\E(m)$ generates the group $ma(m)\Z$. 
We call 
$\E'(m) := \frac{1}{ma(m)}\E(m)$ the 
{\em generic exponent monoid of tensors} (of the cubic format $m$)
and its minimal positive element
$e'(m) :=\frac{1}{ma(m)}\e(m)$ the 
{\em generic minimal exponent of tensors}.

It is a well-known fact, e.g., see~\cite{lama:04}, that 
$$
 \dim \Oh(\ot^3\C^m)^{\SL^3_m}_{m\d} = k_m(\d)  ,
$$
where 
$k_m(\d) := k(m\times \d,m\times \d,m\times \d)$ 
is the {\em Kronecker coefficient} assigned to three partitions of 
the same rectangular shape $m\times \d := (\d,\ldots,\d)$ ($m$ times).  
We therefore obtain that 
\begin{equation}\label{eq:kmd}
 \E'(m) = \big\{ \d\in\N \mid k_m(\d) >0 \big\} \quad \mbox{if $m>2$.} 
\end{equation}

\begin{remark}
The case $m=2$ is special since $a(2)=2$. Here we have 
$\big\{ \d\in\N \mid k_2(\d) >0 \big\} = 2\E'(2)= 2\N$.
In fact, it is known that $k_2(\d)$ equals $1$ if $\d$ is even and 
$k_2(\d)=0$ otherwise; cf.~\cite{rewe:94}. 
\end{remark}

\begin{example}\label{ex:table}
The generic exponent monoid $\E'(3)=\{ 0,2,3,4,\ldots\}$ is generated by $2,3$. 
Thus $1$ is the only gap. 
The first function values of $k_3$ are given by 
$$
\begin{array}{l|ccccccccccccc}
\d &        0&1 &2&3&4&5&6&7&8&9&10&11&12\\ \hline 
k_3(\d) &1&0&1&1&2&1&3&2&4&3&5&4&7\\
\end{array}
$$
The above table is consistent with the fact that 
$\Oh(\otimes^3\C^3)^{\SL_3^3}$ is generated by three 
homogeneous generators having the degrees $6,9,12$, 
see \cite{Vin:76, BHO:14}.
We also remark that $k_3(3)=1$ states the existence and 
uniqueness of Strassen's invariant~\cite{stra:83}. 
It is instructive to verify here the monotonicity relation 
$k_m(\d) \le k_m(\d+\ell)$, holding if $\ell\in \E'(m)$. 
(This property follows from the general observation
$k(\la,\mu,\nu) \le k(\la+\a,\mu+\b,\nu+\g)$ if $k(\a,\b,\g)>0$, 
resulting from interpreting the Kronecker coefficients as the dimensions 
of highest weight vector spaces; cf.~\cite{manivel:11}.) 
\end{example}

The following computations were obtained with the {\sc Derksen} program 
for computing Kronecker coefficients.
(This is an adaption by J.~H\"uttenhain of a code originally written by H.~Derksen.) 

\begin{example}\label{ex:E3}
We have 
$\E'(4)=\E'(3)=\{0,2,3,4,\ldots\}$ generated by $2,3$, with the only gap~$1$. 
Further, we have 
$\E'(5)=\E'(6)=\E'(8)=\E'(9)=\{0,3,4,5,6,\ldots\}$ 
generated by $3,4,5$, with the gaps $1,2$. 
But we have 
$\E'(7)=\{0,4,5,6,7,8,\ldots\}$ 
generated by $4,5,6,7$, with the gaps $1,2,3$. 
This show that
$e'(3)=e'(4)=2$, $e'(5)=e'(6)=e'(8)=e'(9)=3$, $e'(7)=4$. 
Computing a little further, we obtain 
$\E'(10) = \E'(11) = \E'(12) = \{ 0,4,5,6,7,\ldots\}$ and hence 
$e'(10)=e'(11)= e'(12) = 4$. 
Further, $e'(13)=e'(14)=e'(15)=e'(16)=4$.
\end{example}

We can now define the key objects of study for tensors, as in Section~\ref{se:FI-forms}. 

\begin{definition}\label{def:reg-T}
By the {\em fundamental invariant}~$\Phi_w$ of $w\in\ot^3\C^m$  
we understand the $\SL^3_m$-invariant $\Phi_w$ in $\Oh(\ol{Gw})$ of the 
(minimal) degree $e(w)$ satisfying $\Phi_w(w)=1$. 
\end{definition}

We note that $\Phi_w = (\phi_w)^{e'(w)}$ on $Gw$ by Lemma~\ref{le:1-dim-T}.

\begin{theorem}\label{th:localization-tensor}
\begin{enumerate}
\item The zero set of $\Phi_w$ in $\ol{Gw}$ equals the boundary $\ol{Gw}\setminus Gw$. 
\item The coordinate ring of the orbit $Gw$ equals the localization of 
the coordinate of the orbit closure $\ol{Gw}$ with respect to the fundamental invariant $\Phi_w$.
\item  If $ma(w) < e(w)$, that is, $1<e'(w)$, then the vanishing ideal of the boundary $\ol{Gw}\setminus Gw$ in $\ol{Gw}$
is strictly larger than the principal ideal $\Phi_w\Oh(\ol{Gw})$. 
Moreover, $\ol{Gw}$ is not a normal algebraic variety. 
In particular, it is not a valuation ring. 
\end{enumerate}
\end{theorem}

The proof 
is completely analogous to the one of 
Proposition~\ref{le:ext-zero} and Theorem~\ref{le:Phi_w-zeroset} 
and therefore omitted.

\subsection{Generic fundamental invariant of tensors}
\label{se:gen-fun-ten}

We investigate the generic minimal exponent $e'(m)$.

\begin{theorem}\label{pro:km}
\begin{enumerate}
\item We have $\lceil\sqrt{m}\rceil \le e'(m)$ if $m>2$. 
\item We have $e'(m) \le m$ if the Alon-Tarsi conjecture holds for $m$. 
\item If $m=n^2$, then $k_{n^2}(n) = 1$. In particular, $e'(n^2)=n$. 
\end{enumerate}
\end{theorem}

\begin{proof}
1. According to \eqref{eq:kmd}, $e'(m)$ is the minimum $\d$ such that $k_m(\d)>0$
(recall $m>2$). 
The symmetry property of Kronecker coefficients allows to transpose two partitions 
without changing the value: 
$k(m\times \d,m\times \d,m\times \d) = k(m\times \d,\d\times m,\d\times m)$.
Moreover, it is known that $k(\la,\mu,\nu)=0$ if 
$\ell(\la) > \ell (\mu)\ell(\nu)$. 
Hence  $k_m(\d) =0$ if $m > \d^2$.
This implies $e'(m)\ge \lceil\sqrt{m}\rceil$. 

2. We construct an $\SL_m^3$-invariant $P$ in $\Oh(\otimes^3\IC^m)_{m^2}$, we evaluate it at the unit tensor $\lan m\ran$,
and we show that $P(\lan m\ran)$ is the difference between the number of column-even and the number of column-odd Latin squares of order $m$.

Let $\{e_i\}_{1\leq i\leq m}$ be the standard basis of $\IC^m$.
We observe that $\zeta:=(\bigwedge_{i=1}^m e_i)^{\otimes m}$ is an $\SL_m$-invariant in $\otimes^m(\otimes^m(\IC^m))$.
We define $\zeta^t$ to be the image under the canonical automorphism of
$\otimes^m(\otimes^m(\IC^m))$ that sends $v_{1,1} \otimes v_{1,2} \otimes \cdots \otimes v_{m,m}$
to $v_{1,1} \otimes v_{2,1} \otimes \cdots \otimes v_{m,m}$ (``transposition'').
We interpret both $\zeta$ and $\zeta^t$ as elements in $\otimes^{m^2}\IC^m$.
Now $\zeta \otimes \zeta^t \otimes \zeta^t$ is an $\SL_m^3$-invariant in $\otimes^3(\otimes^{m^2}(\IC^m))$.
We use the canonical isomorphism $\otimes^3(\otimes^{m^2}(\IC^m)) \simeq \otimes^{m^2}(\otimes^3(\IC^m))$
to interpret $\zeta\otimes \zeta^t \otimes \zeta^t$ as an element in $\otimes^{m^2}(\otimes^{3}(\IC^m))$.
Applying the symmetrization map
$\otimes^{m^2}(\otimes^{3}(\IC^m)) \twoheadrightarrow \Oh(\otimes^{3}(\IC^m))_{m^2}$
to $\zeta\otimes \zeta^t \otimes \zeta^t$
(identifying $\IC^m$ with its dual),
we obtain a function $P$ in $\Oh(\otimes^{3}(\IC^m))_{m^2}$.
The principle of polarization and restitution says that $P(w)$ equals the value of the scalar product
\[
\langle \zeta\otimes \zeta^t \otimes \zeta^t , w^{\otimes {m^2}} \rangle.
\]
Setting $w = \lan m \ran$ to be the unit tensor we obtain
\[
w^{\otimes {m^2}} = \sum_{i_{1,1},i_{1,2}\ldots,i_{m,m}=1}^m e_{i_{1,1}}^{\otimes 3} \otimes e_{i_{1,2}}^{\otimes 3} \otimes \cdots \otimes e_{i_{m,m}}^{\otimes 3}
\]
We write
\begin{equation}\label{eq:bigsum}
P(w) = \sum_{i_{1,1},i_{1,2}\ldots,i_{m,m}=1}^m \langle \zeta\otimes \zeta^t \otimes \zeta^t , e_{i_{1,1}}^{\otimes 3} \otimes e_{i_{1,2}}^{\otimes 3} \otimes \cdots \otimes e_{i_{m,m}}^{\otimes 3}\rangle.
\end{equation}
Using the special structure of $\zeta$ we see that
each contraction of $\zeta\otimes \zeta^t \otimes \zeta^t$ with a tensor of the form
$e_{i_{1,1}}^{\otimes 3} \otimes e_{i_{1,2}}^{\otimes 3} \otimes \cdots \otimes e_{i_{m,m}}^{\otimes 3}$
is a product of determinants of $3m$ matrices $A_1,A_2,\ldots,A_m,A_1,A_2,\ldots,A_m,B_1,B_2,\ldots,B_m \in \IC^{m \times m}$.
For $1 \leq a \leq m$ the matrix $A_a$ is the concatenation of the column vectors $e_{i_{b,a}} \in \IC^m$ for $1 \leq b \leq m$.
For $1 \leq a \leq m$ the matrix $B_a$ is the concatenation of the column vectors $e_{i_{a,b}} \in \IC^m$ for $1 \leq b \leq m$.
Note that all matrices $A_a$ and $B_a$ only have entries from $\{0,1\}$.
In order for a summand in \eqref{eq:bigsum} to be nonzero there can be no double columns in any of the matrices $A_a$ or $B_a$
(so all $A_a$ and $B_a$ must be permutation matrices), which can be rephrased as:
the numbers $1 \leq i_{1,1},i_{1,2},\ldots,i_{m,m} \leq m$ must form a Latin square~$T$.
Now $\big(\prod_{a=1}^{m}\det(A_a)\big)\big(\prod_{a=1}^{m}\det(A_a)\big)\big(\prod_{a=1}^{m}\det(B_a)\big)=\prod_{a=1}^{m}\det(B_a)$,
which is the product of the signs of the permutations in the columns of~$T$.
Therefore $P(w)$ is the difference between the number of column-even and the number of column-odd Latin squares of order $m$.

3. For the third statement, we need to show that 
$k(n^2\ti n,n^2\ti n,n^2\ti n)=1$. 
We use the following symmetry relation
from \cite[Cor.~4.4.15]{ike:12b}:
\begin{equation}\label{eq:kron-sym}
 k(\la,\mu,\nu) = k(n^2\times n - \la, n^2\times n - \mu,n^2\times n - \nu) ,
\end{equation}
holding for any partitions $\la,\mu,\nu$ contained in $n^2 \times n$. 
(Here, $n^2\times n - \la$ denotes the partition corresponding to the complement 
of the Young diagram of $\la$ in the rectangle $n^2 \times n$.)
In particular, setting $\la=\mu=\nu= n^2 \times n$, we get 
$k(\la,\mu,\nu)= k(0,0,0)=1$.
\end{proof}

\begin{remark}
The fact $k_{n^2}(n) = 1$ also immediately follows combinatorially from the upper and lower bounds on the Kronecker coefficient in 
\cite{Man:97} and \cite{BI:13}, see also \cite[Lemma 2.3]{IMW:15}.
\end{remark}

If $m=n^2$ is a square, then Theorem~\ref{pro:km}(3) 
states that, up to a scaling factor, 
there is exactly one homogeneous $\SL^3_m$-invariant 
$F_{n}\colon \ot^3 \C^m \to \C$  of degree~$e(m)=ma(m)e'(m)=n^3$
(and no nonzero invariant of smaller degree). 
We call $F_{n}$ the {\em generic fundamental invariant} of $\ot^3\C^{n^2}$. 

Before providing an explicit formula for $F_{n}$, we state our main general 
result concerning tensors.
The proof is analogous to Theorem~\ref{th:newmain} 
and therefore omitted.

\begin{theorem}\label{th:newmain-T}
Assume $w\in \ot^3\C^m$ is polystable. Then 
$\sqrt{m}/a(w) \le e'(w)$ if $m>2$. 
Moreover, if $m=n^2$, we have equality iff $F_{n}(w) \ne 0$. 
In this case, we have 
$\Phi_w =  F_{n}(w)^{-1} F_{n}$ on $\ol{Gw}$. 
\end{theorem}

The following is immediate. Note that the orbit closure of almost all $w\in\ot^3\C^2$ 
equals $\ot^3\C^2$ and hence is normal.

\begin{corollary}\label{cor:not-normal-generic-tensor}
1. $\ol{Gw}$ is not normal if $a(w) < \sqrt{m}$. 

2. Let $m\ge 3$. Then $\ol{Gw}$ is not normal for almost all $w\in\ot^3\C^m$.  
\end{corollary}

In the following, we assume $m=n^2$, and provide 
the promised explicit description of the invariant~$F_{n}$. 

Consider the combinatorial cube $[n]^3$. 
We shall call the elements $p\in [n]^3$ {\em points}. 
For $\ell\in [n]$ we call 
$\{ (i,j,k) \in [n]^3 \mid i= \ell\}$ its $\ell$th slice in $x$-direction.  
(In other words, we fix the first coordinate to $\ell$.) 
Similarly, we define its $y$-slices by fixing the second coordinate, and its 
$z$-slices by fixing the third coordinate. 

We fix a total ordering of $[n]^3$ once and for all. 
This defines an ordering of each of the slices of~$[n]^3$. 
By a {\em labeling} we shall understand a map $\a\colon [n]^3\to [n^2]$. 
Suppose that a labeling defines a bijection on each $x$-slice. 
Then the resulting permutation of each $x$-slice has a well-defined sign.
We define $\sgn_x(\a)\in\{-1,1\}$ of the labeling $\a$ to be the product of the signs of 
the permutations of all $x$-slices. If the labeling $\a$ fails to be a bijection on an $x$-slice, 
we write $\sgn_x(\a):=0$. 
Similarly, we define $\sgn_y(\a)$ and $\sgn_z(\a)$. 

We express a tensor $w=\sum_{abc} w_{abc}\, \ket{abc} \in \ot^3\C^{n^2}$ 
in the standard basis $\ket{abc}$ with coordinates $w_{abc}\in\C$. 
Consider the following homogeneous polynomial  
$F_n\colon\ot^3 \C^{n^2}\to\C$ of degree~$n^3$ defined by 
\begin{equation}\label{eq:fund-invar-tensor} 
F_n(w) := \sum_{\a,\b,\g\colon [n]^3\to [n^2]} \sgn_x(\a)\sgn_y(\b)\sgn_z(\g)\, \prod_{p\in [n]^3} w_{\a(p)\b(p)\g(p)} ,
\end{equation}
where the sum is over all triples $\a,\b,\g$ of labelings $[n]^3\to [n^2]$. 

The following result shows that $F_n$ is indeed the invariant we are looking for;  
see also \cite[Example 4.12]{BI:13}. 

\begin{theorem}\label{th:fund-invar-tensor}
We have 
$F_n((g_1\ot g_2\ot g_3)w) = (\det(g_1)\det(g_2)\det(g_3))^n \, F_n(w)$ for all 
$(g_1,g_2,g_3)\in\GL_{n^2}$ and $w\in\ot^3\C^{n^2}$. 
Moreover, $F_n\ne 0$.
\end{theorem}

\begin{proof}
The proof is rather indirect. Let $d=n^3$ and $m=n^2$.
We will explicitly construct a generating set for the space of $\SL_m^3$-invariants in $\Oh(\otimes^3\IC^m)_d$
and then see that this vector space is \emph{at most} 1-dimensional with $F_n$ being the only element up to scale.
From Proposition~\ref{pro:km}(3) we know that $k_{n^2}(n)=1$, therefore $F_n \neq 0$.

The following construction arises by spezializing the description of ``obstruction designs'' in \cite{ike:12b} and \cite{BI:13}
from highest weight vectors to invariants. We refer to these papers for more references.

We start by constructing a generating set for the $\SL_m^3$-invariants in $\otimes^d(\otimes^3\IC^m)$.
Let $(\otimes^d(\otimes^3\IC^m))^{\SL_m^3} \subseteq \otimes^d(\otimes^3\IC^m)$ denote the $\SL_m^3$-invariant subspace.
Clearly the actions of $S_d^3$ and $\SL_m^3$ on $\otimes^d(\otimes^3\IC^m)$ commute.
By Schur-Weyl duality, the $S_d^3$-representation $(\otimes^d(\otimes^3\IC^m))^{\SL_m^3}$ is irreducible.
Therefore, if we have a nonzero element $P \in (\otimes^d(\otimes^3\IC^m))^{\SL_m^3}$, we know that
$\{\pi P \mid \pi \in S_d^3\}$ is a generating set of $(\otimes^d(\otimes^3\IC^m))^{\SL_m^3}$.
Let $\{e_1,\ldots,e_m\}$ be the standard basis of $\IC^m$.
We consider the $\SL_m$-invariant $\zeta := (\bigwedge_{i=1}^{m} e_i)^{\otimes n} \in \otimes^{d}\IC^m$
and put $P := \zeta \otimes \zeta \otimes \zeta \in (\otimes^3(\otimes^d\IC^m))^{\SL_m^3}$.
Using the canonical isomorphism $\otimes^3(\otimes^d\IC^m) \simeq \otimes^d(\otimes^3\IC^m)$
we interpret $P \in (\otimes^d(\otimes^3\IC^m))^{\SL_m^3}$.
Therefore $(\otimes^d(\otimes^3\IC^m))^{\SL_m^3}$ is generated by $\{\pi P \mid \pi \in S_d^3\}$.
The vector $\zeta$ wedges together basis vectors in tuples of cardinality~$m$.
We want to describe this combinatorially by saying that $\zeta$ has the following wedge list:
\[
(1,2,\ldots,m),(m+1,m+2,\ldots,2m),\ldots,((n-1)m+1,\ldots,d).
\]
We call \emph{blocks} the subsets in which the parantheses partition the set $\{1,\ldots,d\}$.
The combinatorial description of $P$ is given by a triple $t$ of wedge lists
and the action of $S_d^3$ permutes the entries in these lists.
We denote by $P_t \in (\otimes^d(\otimes^3\IC^m))^{\SL_m^3}$ the invariant corresponding to~$t$.
The structure of $\zeta$ implies that if we permute whole blocks in $t$ while preserving the order inside the blocks,
$P_t$ does not change.
If only entries inside a block are changed, then, up to sign, $P_t$ does not change.
So we see that $P_t$ can be combinatorially defined (up to a sign) by a triple of set partitions of the set $\{1,\ldots,d\}$
into $n$ sets of cardinality~$m$.
A crucial observation is the following: If there exist two numbers $i$ and $j$ whose block
coincides in the first wedge list of~$t$ and whose block also coincides in the second and third wedge list,
then $P_t$ vanishes under the projection $\Psi:\otimes^d(\otimes^3\IC^m)\twoheadrightarrow \Oh(\otimes^3\IC^m)_d$
given by $\Psi := \frac 1 {d!}\sum_{\sigma \in S_d} \sigma$.
The reason is that if $S_2$ is the symmetric group on $\{i,j\}$, then
$P_t$ vanishes already under the projection $\frac 1 {2}\sum_{\sigma \in S_2} \sigma$,
because $\frac 1 2\sum_{\sigma \in S_2} \sigma P_t = \frac 1 2 (P_t-P_t) = 0$.
We call this the \emph{zero pattern}.
The main observation which we prove next is that there exist only very few triples of set partitions
of the set $\{1,\ldots,d\}$ into $n$ sets of cardinality~$m$ that avoid the zero pattern.
Indeed, these triples all coincide up to the action of $S_d$, which renames the numbers in all three wedge lists simultaneously.
Therefore the corresponding $P_t$ all get mapped to the same element under~$\Psi$.

The argument in \cite[Claim~7.2.17]{ike:12b} or \cite[Exa.~4.12]{BI:13} is as follows.
From a triple of set partitions that avoid the zero pattern we can construct
a cardinality $d$ subset of $\N^3$ as follows:
In each of the three wedge lists we order the blocks from $1$ to $n$ in an arbitrary way.
Now for each number $i\in\{1,\ldots,d\}$ we let the $x$-coordinate of the $i$th point
be the number of its block in the first wedge list.
The $y$ coordinate shall be the number of the block in the second wedge list
and the $z$ coordinate shall be the number of the block in the third wedge list.
No point is repeated because the wedge list triple avoids the zero pattern.
This construction can be reversed in the obvious way, but the reverse process is only unique
up to the action of $S_d$.
Since there are only $n$ blocks in each wedge list this construction results in a subset of
the $n \times n \times n$ cube.
But since we constructed $d=n^3$ points, the point set is \emph{precisely} the $n \times n \times n$ cube.
Let $t$ be a wedge list triples that avoids the zero pattern.
We have seen that $\Oh(\otimes^3\IC^m)_d^{\SL_m^3}$ is at most 1-dimensional and generated by $\Psi(P_t)$.
If we evaluate $\Psi(P_t)$ at a point $w$, i.e.,
if we contract $P_t$ with a tensor $w^{\otimes d}$,
we get precisely the sum in equation~\eqref{eq:fund-invar-tensor},
hence $(\Psi(P_t))(w)=F_n(w)$.
With Proposition~\ref{pro:km}(3) it follows that $F_n \neq 0$.
\end{proof}

As for Lemma~\ref{le:irred}, 
we can show that the generic fundamental invariant $F_n$ is an irreducible polynomial.

\begin{remark}
The above has a natural generalization to noncubic formats. 
Let $n_1,n_2,n_3\ge 1$. We know from~\cite[Cor.~4.4.15]{ike:12b} that 
$k(n_2n_3\times n_1, n_1n_3\times n_2, n_1n_2\times n_3)=1$.
The corresponding invariant $F_{n_1n_2n_3}$ of degree $n_1n_2n_3$ evaluated at a tensor 
$w=\sum_{abc} w_{abc}\, \ket{abc} \in \C^{n_2n_3} \ot \C^{n_1n_3}\ot \C^{n_1n_2}$ 
is given as in~\eqref{eq:fund-invar-tensor}, 
where the sum is over the labelings 
$\a\colon P \to [n_2n_3], \b\colon P \to [n_1n_3], \g\colon P \to [n_1n_2]$
with $P:=[n_1]\times [n_2]\times [n_3]$.
In the special case $n_2=n_3=1$ we get 
$F_{n11}(w)=n!\, \det (w)$, when we interpret
$w\in\C^{1} \ot \C^{n}\ot \C^{n}$ as a matrix.
\end{remark}

We finish this section with a couple of interesting problems related with $\E(m)$. 

\begin{problem}
Give a direct proof of $F_n\ne 0$ by evaluating $F_n$ at a (generic) $w\in\ot^3\C^m$. 
\end{problem}

\begin{problem}\label{pr:compl-Fn}
What is the computational complexity to evaluate $F_{n}$?
\end{problem}

\begin{problem}\label{pr:kmd-pos}
How close is $e'(m)$ to $\sqrt{m}$ if $m$ is not a square? 
\end{problem}

\begin{remark}
We note that $e'(m) = \lceil \sqrt{m} \rceil$ is false for $m=7$, 
which is the only counterexample for $m\le 16$.
We have $e'(7)=4$. 
We further note that a computation showed that $k_m(e'(m))=1$ for 
$m=2,3,4,5,6,8,9,14,15,16$, but $k_7(e'(7))=k_7(4)=14$, and 
$k_{10}(e'(10))= k_{10}(4) =13$, $k_{11}(e'(11))= k_{11}(4)=6$,
$k_{12}(e'(12))= 5$, $k_{13}(e'(13))= 2$.  
So in general, we do not have $k_m(e'(m))=1$. 
\end{remark}

\begin{problem}\label{pr:kmd-holes}
Do we have 
$\E'(m)=\{0\} \cup (e(m) + \N$) if $m>2$?
In other words, are $1,2,\ldots,e(m)-1$ the only possible gaps? 
(We have verified this for $m\le 12$.)  
\end{problem}

\begin{problem}
Determine the asymptotic growth of~$k_m$. More specifically, 
find $c_m$ and $\mu_m$ such that 
$k_m(\delta) \sim c_m \delta^{\mu_m}$ for $\delta\to\infty$. 
\end{problem}

In Example~\ref{ex:table} $k_3(\delta)$ is a quasipolynomial of period 12
and $k_3(\delta) \sim \frac 1 {48} \delta^2$ for $\delta\to\infty$, see the OEIS sequence A005044,
where all 12 polynomials are specified:
$\delta^2/48, (\delta^2 + 6\delta - 7)/48, (\delta^2 - 4)/48, (\delta^2 + 6\delta + 21)/48, (\delta^2 - 16)/48, (\delta^2 + 6\delta - 7)/48, (\delta^2 + 12)/48, (\delta^2 + 6\delta + 5)/48, (\delta^2 - 16)/48, (\delta^2 + 6\delta + 9)/48, (\delta^2 - 4)/48, (\delta^2 + 6\delta + 5)/48$.

\subsection{Minimal degree for specific tensors}

We can combinatorially characterize the minimal exponent 
in a few interesting cases. 

To evaluate $F_n$ at the unit tensor 
$\lan n \ran = \sum_{a} \ket{aaa} \in \ot^3\C^{n^2}$,  
we introduce the following 3D generalization of latin squares.
Recall that we fixed a total ordering of $[n]^3$.

\begin{definition}
A {\em Latin cube} of size~$n$ is a map 
$\a\colon [n]^3\to [n^2]$ 
that is a bijection on each of the $x$-slices, $y$-slices, 
and $z$-slices of the combinatorial cube $[n]^3$. 
The latin cube is called {\em even} if the product of the signs 
of the resulting permutations of all $x$-slices, $y$-slices, and $z$-slices 
equals one. Otherwise, the latin cube is called {\em odd}.
\end{definition}

\begin{proposition}\label{le:Fn-unit-tensor}
1. $F_n(\lan n^2\ran)$ equals the difference between the number of 
even Latin cubes of size~$n$ and the number of odd Latin cubes of size~$n$. 

2. We have $F_n(\lan n^2\ran) = 0$ if $n$ is odd.
\end{proposition}

\begin{proof}
The first assertion is obvious from \eqref{eq:fund-invar-tensor}. 

For the second assertion, 
we consider $[n^2]$ as the set of symbols. Exchanging two fixed symbols 
(e.g., 1 with 2) defines an involution $\a\mapsto\a'$ of the set of Latin cubes of size~$n$. 
We have $\sgn(\a') = (-1)^{3n} \sgn(\a)= - \sgn(\a)$, since the involution creates 
in each slice a transposition, there is a total of $3n$ slices, and $n$ is odd.
\end{proof}

The following question, which is analogous to the Alon-Tarsi conjecture, 
is important for understanding tensor border rank. 
We have verified this in the cases $n=2$ and $n=4$.

\begin{problem}\label{pr:3D-AT}
Let $n$ be even.
Is the number of even Latin cubes of size~$n$ different from the number of odd Latin cubes of size~$n$?
\end{problem}

The matrix multiplication tensor 
$\lan n,n,n\ran = \sum_{ijk} \ket{(ij)(jk)(ki)} \in \ot^3\C^{n\times n}$ 
has the dimension format $(n^2,n^2,n^2)$. 
We are going to describe the evaluation of $F_n$ at the matrix multiplication tensor.   
A labeling $\a\colon [n]^3\to [n]^2$ is given by its two 
coordinate maps $\mu, \nu\colon [n]^3\to [n]$, i.e., 
$\a(p)=(\mu(p),\nu(p))$. 
The following characterization is obvious from~\eqref{eq:fund-invar-tensor} 
and the structure of the matrix multiplication tensor. 

\begin{proposition}\label{le:eval-Fn-mamu}  
We have 
$$
 F_n(\lan n,n,n\ran) = \sum_{\mu,\nu,\pi\colon [n]^3\to [n]} \sgn_x(\mu,\nu)\sgn_y(\nu,\pi)\sgn_z(\pi,\mu) .
$$ 
\end{proposition}

\begin{corollary}\label{cor:e-unit-tensor}
\begin{enumerate}
\item We have $e'(\lan m\ran) \ge \frac12 \lceil \sqrt{m} \rceil$ if $m>1$.
\item Let $m=n^2$ be even. Then $e'(\lan n^2\ran) = n/2$ iff the 
the number of even Latin cubes of size~$n$ is different from the number of odd Latin cubes of size~$n$.
\item We have $e'(\lan n,n,n\ran) \ge n$, with equality holding iff $F_n(\lan n,n,n\ran) \ne 0$. 
(See Lemma~\ref{le:eval-Fn-mamu} for a combinatorial expression for this condition.)
\end{enumerate}
\end{corollary}

\begin{proof}
This follows from Theorem~\ref{th:newmain-T}
and the determination of the periods in 
Corollary~\ref{cor:stabut} and Corollary~\ref{cor:period-mamu}.
\end{proof}

Theorem~\ref{th:newmain-T} immediately implies the following. 
This result already appeared in \cite{BI:10}. 

\begin{corollary}\label{cor:not-normal-unit-mamu}
The orbit closure of the unit tensor $\lan m\ran$ is not normal if $m \ge 5$.
The orbit closure of the matrix multiplication tensor $\lan n,n,n\ran$ is not normal if $n\ge 2$.
\end{corollary}

\section{Proofs for Section~\ref{se:exp-monoid-forms} and Section~\ref{se:FI-tensors}}
\label{se:pf-group_S_w} 

\subsection{Algebraic curves with dense $\Ct$-orbit}

We present here a general observation. 
Let $Z$ be  an affine algebraic curve over $\C$ 
and assume the group $\Ct$ acts rationally on $Z$.  
We further assume that there exists $p\in Z$ such that 
the orbit morphism 
$\sigma\colon\Ct\to Z, t\mapsto t p$ has  
a regular extension to a surjective morphism 
$\sigma\colon\C\to Z$. 
In particular, $Z$ is irreducible.
Note that $\Ct$ acts nontrivially on $Z$
(otherwise, $Z$ would be a point). 
We consider now 
$\sigma^*\colon\Oh(Z) \hookrightarrow \Oh(\C) =\C[T]$ and write 
\begin{equation}\label{eq:OhZ}
 \sigma^*(\Oh(Z)) = \bigoplus_{k\in S} \C T^k
\end{equation}
with a submonoid $S\subseteq\N$. 
The stabilizer $\{ t\in\Ct \mid t p = p \}$ of~$p$ 
is a finite subgroup of $\Ct$ 
and thus equals $\mu_a$ for some $a\ge 1$. 

\begin{lemma}\label{le:S-grp}
The monoid $S$ generates the subgroup $a\Z$ of $\Z$.
\end{lemma}

\begin{proof}

Suppose that $S$ generates the subgroup $\ell\Z$ of $\Z$, 
where $\ell\ge 1$. 
Then $\sigma^*(\Oh(Z)) \subseteq \C[T^\ell]$ 
and hence $\sigma^*$ factors into 
the chain of ring homomorphisms 
$\Oh(Z) \to \C[T^\ell] \hookrightarrow \C[T]$. 
This yields the factorization of $\sigma$ into 
the chain of regular maps 
$\C \stackrel{\phi}{\to}\C \stackrel{\tilde{\sigma}}{\to} Z$, 
where $\phi(t) = t^\ell$.  
This implies 
$\mu_\ell \subseteq \{t\in\Ct \mid tp =p \} = \mu_a$, 
hence $\ell$ divides $a$. 

For the converse, we observe that $\sigma$ factors 
via a map $\tilde{\sigma}\colon\C\to Z$ as 
$\sigma = \tilde{\sigma}\circ\phi$, where 
$\phi\colon\C\to\C$ is given by $\phi(t) = t^a$. 
The map $\tilde{\sigma}\colon\Ct\to Z$ is regular by 
the universal property of the quotient 
$\Ct\stackrel{\phi}{\to}\Ct/\mu_a$ of algebraic groups
provided by $\phi$. Moreover, $\tilde{\sigma}$ is continuous at~$0$ 
because $\sigma$ is so. It follows that $\tilde{\sigma}$ is regular on~$\C$ 
(e.g., because $\C$ is normal). 
We deduce that $\sigma^*=\phi^*\circ\tilde{\sigma}^*$, 
and hence  $\sigma^*(\Oh(Z)) \subseteq \C[T^a]$. 
Since $\ell\Z$ is the subgroup generated by $S$, we get 
$\ell\Z\subseteq a\Z$, hence $a$ divides $\ell$.
We have thus shown that $\ell=a$. 
\end{proof}

\subsection{Proof of Theorem~\ref{th:group_S_w}}
\label{se:pf_group_S_w}

Let us recall the situation. 
We have a group embedding $\iota\colon\Ct\to G=\GL_m,$ via $\iota(t)=t\Id_m$. 
Moreover, $G$ is the product of the normal subgroup $\Ct\Id_m$ and the subgroup $G_s:=\SL_m$. 
Recall the action of $G$ on $W:=\Sym^D\C^m$ and note that 
\begin{equation}\label{eq:iota(t)}
 \iota(t)w = t^D w \quad \mbox{ for $t\in\Ct, w\in \Sym^D\C^m$} .
\end{equation}
This means that $t\in \Ct$, when seen as embedded in $G$ via $\iota$, 
acts by the $D$th power of $t$. 
From this it follows that the orbit closure $X:=\ol{Gw}$ is invariant under scalar multiplication 
and $0\in X$. 
The scalar multiplication on $W$ defines a $\Ct$-action on $X$, which commutes with the 
$G_s$-action. Hence we get an induced $\Ct$-action on $\Oh(X)$, which commutes with 
the $G_s$-action as well. Therefore, the subring $\Oh(X)^{G_s}$ of $G_s$-invariants is $\Ct$-invariant. 

For the following facts from geometric invariant theory, see \cite{kraf:84}. 
Since $G_s$ is reductive, the ring $\Oh(X)^{G_s}$ is a finitely generated $\C$-algebra and 
therefore it can be interpreted as the coordinate ring of an affine variety~$Z$. 
The inclusion $\Oh(Z) \hookrightarrow \Oh(X)$ defines 
a morphism $\pi\colon X\to Z$ of algebraic varieties,
which is called the {\em GIT-quotient} of $X$ by the 
action of the group $G_s$. 
We have a $\Ct$-action on $Z$, defined by the $\Ct$-action on $\Oh(Z)$, 
and $\pi$ is $\Ct$-equivariant. 

By definition, $\pi$ is constant on $G_s$-orbits.
It is an important property of GIT-quotients that disjoint orbit closures are 
mapped to distinct points under $\pi$.  
It also known that $\pi$ is surjective, cf.~\cite[II.\S3.2]{kraf:84}. 
Consider the orbit map of $p:=\pi(w)$ 
\begin{equation}\label{eq:def-sigma}
 \sigma\colon\Ct\to Z,\ t\mapsto tp =\pi(tw) ,
\end{equation}
which has a regular extension $\sigma\colon\C\to Z$ given by $\sigma(0):=\pi(0)$.

It  will be convenient to use the notation
$X^o := Gw$ and $\partial X := X\setminus X^o$. 
Note that $0\in\partial X$. 

\begin{lemma}\label{le:sigma_image}
$Z$ is an irreducible affine curve and the map $\sigma\colon\C\to Z$ is surjective. 
Moreover, $\pi(X^o)=\sigma(\Ct)$, 
$\pi(\partial X) = \{ \pi(0)\}$, and 
$p\not\in\sigma(\Ct)$. 
\end{lemma}

\begin{proof}
(1) Any $g\in G$ can be written as $g=\iota(t)g_s$,  
where $t\in\Ct$ and $g_s\in G_s$. 
By~\eqref{eq:iota(t)}, we have $gw = t^D g_s w$, and hence 
$$
 \pi(gw) = t ^D\pi(g_s w) = t^D \pi(w) = \sigma(t^D) ,
$$
hence $\pi(X^o)\subseteq \sigma(\Ct)$. 
The other inclusion is obvious.

(2) 
Let $u\in\partial X$, say $u=\lim_{n\to\infty} g_n w$ 
for some sequence $g_n$ in $G$. 
Write $g_n = \iota(t_n)\tilde{g}_n$, where $\tilde{g}_n\in G_s$.
Then, $\det(g_n) = \det(\iota(t_n)) = t_n^m $. 
We have $\lim_{n\to\infty} t_n = 0$ since $\lim_{n\to\infty} \det(g_n) = 0$
by Proposition~\ref{le:ext-zero}. 
We obtain 
$$
 \pi(g_n w) = \pi(t_n^D\tilde{g}_n w) = t_n^D \pi(\tilde{g}_n w) 
  = t_n^D \pi(w) = \pi(t_n^D w) .
$$
Hence,
$\lim_{n\to\infty} \pi(g_n w) = \pi(0)$.
On the other hand, $\lim_{n\to\infty} \pi(g_n w) = \pi(u)$ by the 
continuity of~$\pi$.  It follows that $\pi(u)=\pi(0)$. 
This shows that $\pi(\partial X) =\{\pi(0)\}$. 

(3) Suppose by way of contradiction that $\pi(0)\in\sigma(\Ct)$, say 
$\pi(0) = \pi(tw)$ for some $t\in\Ct$. Then the closures of 
the orbits $G_s0$ and $G_s tw$ must intersect. But 
$G_s0=\{0\}$ and $G_s tw = t G_sw$ is closed 
since $G_s w$ is closed by assumption. 
Therefore, $0\in G_s tw$, which implies the 
contradiction $w=0$.

(4) We obtain 
$Z=\pi(X) = \pi(X^o) \cup \pi(\partial X) = \sigma(\Ct) \cup\{\pi(0)\}$.
Hence $\sigma\colon\C\to Z$ is surjective. 
\end{proof} 

We now determine the stabilizer of $p$ under the $\Ct$-action on $Z$.
Recall the degree period~$b(w)$ from Definition~\ref{def:degree-period}.  

\begin{lemma}\label{le:stab-piw}
We have $\{t\in\Ct \mid tp= p\} = \mu_{b(w)}$. 
\end{lemma}

\begin{proof} 
Suppose that $\pi(tw)=\pi(w)$. Then, since $\pi$ is a GIT quotient, 
the orbit closures of $tw$ and $w$ intersect; 
cf.~\cite[II.\S3.2]{kraf:84}. 
Since $G_sw$ is closed, this means that $tw\in G_sw$. 
Therefore, there exists $g_s\in G_s$ such that $tw=g_s^{-1}w$, 
hence $g_s tw = w$.  
If we choose $t_1\in\Ct$ such that $t=t_1^D$, 
then $g_s t w= g_s\iota(t_1)w$, hence  
$g_s \iota(t_1)\in \stab(w)$. 
By the definition of the period~$a(w)$,
using \eqref{def:degree-period}, 
we obtain 
$$
1 = \det(g_s \iota(t_1))^{a(w)} = \det(\iota(t_1))^{a(w)} = t_1^{m a(w)} =  t_1^{D b(w)} = t^{b(w)}.
$$
We have shown that 
$\{t\in\Ct \mid tp= p\} \subseteq \mu_{b(w)}$. 

For the converse, assume that $t\in \mu_{b(w)}$. 
Choose $t_1$ with $t=t_1^D$. Then, similarly as before, 
$1=t^{b(w)} = t_1^{Db(w)} = t_1^{ma(w)}$.
Since $\det(\stab(w)) = \mu_{a(w)}$,  
there exists $g\in\stab(w)$ such that 
$\det(g^{-1}) = t_1^m = \det(\iota(t_1))$. 
Hence $g_s := \iota(t_1) g \in G_s$. 
We obtain 
$$
 g_s w =\iota(t_1) gw = \iota(t_1) w = t_1^{D} w ,
$$
and hence 
$\pi(w) = \pi(g_sw) = t_1^D\pi(w) = t\pi(w)$, that is, $p=tp$. 
\end{proof}

\begin{remark}
The map $\C/\mu_{b(w)}\to Z$, resulting from $\sigma$ by factoring out 
the stabilizer, is the normalization of the curve~$Z$.
\end{remark}

\begin{proof}{(of Theorem~\ref{th:group_S_w})}\ 
Lemma~\ref{le:sigma_image} implies that $Z$ is a curve and 
the orbit map $\sigma\colon\C\to Z,\, t\mapsto tp$
is surjective. 
Moreover, Lemma~\ref{le:stab-piw} 
tells us that the stabilizer of $p$ 
with respect to the $\Ct$-action equals $\mu_{b(w)}$. 
We apply now Lemma~\ref{le:S-grp} to this setting.
Accordingly, we have 
\begin{equation}\label{eq:sOdecomp}
 \sigma^*(\Oh(Z)) = \bigoplus_{k\in S} \C\, T^{kb(w)} ,
\end{equation}
where the submonoid $S\subseteq\N$ generates the group~$\Z$. 
It is sufficient to prove that $S= \E'(w)$. 

Suppose $k\in S$. Then there is $h\in\Oh(Z)$ such that 
$\sigma^*(h) = T^{kb(w)}$. 
For $t\in\Ct$ and $g_s\in G_s$ we have 
$$
 h(\pi(tg_sw))= h(\pi(tw)) = h(\sigma(t)) = t^{kb(w)} .
$$
On the other hand, writing $t=t_1^D$, we get 
(recall $\phi_w(g) = \det(g)^{a(w)}$) 
\begin{equation}\label{eq:det-equality}
 \phi_w(g_stw) = \phi_w(g_s\iota(t_1)w) = \det(g_s\iota(t_1))^{a(w)} = t_1^{ma(w)} 
  = t_1^{Db(w)} = t^{b(w)} .
\end{equation}
Therefore, 
$(\phi_w(g_stw))^k = h(\pi(tg_sw))$ 
and we see that $h\circ\pi$ is a regular extension of $(\phi_w)^k$.
Thus $k\in\E'(w)$ and we have shown that $S\subseteq \E'(w)$. 

For the converse, assume $k\in \E'(w)$. Let $f\in\Oh(X)$ be a 
regular extension of $(\phi_w)^k\colon Gw\to\C$.
Then $f\in\Oh(X)^{G_s}$, hence there exists $h\in\Oh(Z)$ 
such that $h\circ \pi = f$. Since $f$ is homogeneous, 
$h$~is homogeneous as well, say 
$h(t\pi(w)) = c\, t^{sb(w)}$ for $c\in\C$ and $s\in S$. 
In particular, using~\eqref{eq:det-equality}, 
$$
  t^{kb(w)} = (\phi_w(tw))^k = f(tw) = h(\pi(tw)) = c\, t^{sb(w)} , 
$$
for all $t\in\Ct$. Hence $k=s$, which shows $k\in S$ 
and proves the reverse inclusion $\E'(w)\subseteq S$. 
\end{proof}

\subsection{Proof of Theorem~\ref{le:Phi_w-zeroset}}
\label{se:pf_Phi_w-zeroset}

We go back to the setting of the proof of Theorem~\ref{th:group_S_w}. 
Put $X:=\ol{Gw}$ and $\partial X:= X\setminus Gw$.  
Recall the GIT-quotient $\pi\colon X\to Z$,
where $\Oh(Z)\simeq\Oh(X)^{G_s}$.
Lemma~\ref{le:sigma_image} implies that 
$\pi^{-1}(\pi(0))=\partial X$, hence we obtain 
$I(\pi(0))=(\pi^*)^{-1}(I(\partial X))$
for the vanishing ideals.    
Moreover, the map $\sigma\colon \C\to Z$  
defined in \eqref{eq:def-sigma} satisfies 
$\sigma^{-1}(\pi(0)) = \{0\}$, hence 
$I(\pi(0))= (\sigma^*)^{-1}(T\C[T])$. 
Using \eqref{eq:sOdecomp}, we obtain 
\begin{equation}\label{eq:Ipi0}
  \sigma^*(I(\pi(0))) = \bigoplus_{e\in \E(w), e\ne 0} \C\, T^{e} .
\end{equation}
Let $\varphi\in\Oh(Z)$ be such that $\Phi_w=\varphi\circ\pi$. 
Equation~\eqref{eq:det-equality} implies that 
$\varphi\circ\sigma = T^{e(w)}$. 

For the first assertion, suppose that $\Phi_w$ generates the vanishing ideal $I(\partial X)$ in $\Oh(X)$. 
Then $\Phi_w$ generates the ideal $I(\partial X)\cap \Oh(X)^{G_s}$. 
(Indeed, if $h\in\Oh(X)$ is such that $\Phi_w h$ is $G_s$-invariant, 
then $h$ is $G_s$-invariant.) 
That is, $\varphi$ generates the ideal $I(\pi(0))=(\pi^*)^{-1}(I(\partial X))$ in $\Oh(Z)$. 
Hence $T^{e(w)}$ generates the ideal 
$\bigoplus_{e\in \E(w), e\ne 0} \C\, T^{e}$. 
This implies $\E(w)\subseteq e(w)\N$ and hence $e(w)=b(w)$, since 
$\E(w)$ generates the group $b(w)\Z$.  

For the second assertion, note that 
it is a known fact that if $\Oh(X)$ is normal, then its 
ring of invariants $\Oh(X)^{G_s}$ is normal as well, 
e.g., see~\cite[Prop.~3.1]{dolgachev:03}. 
Hence, in the situation from before,  
$Z$ is normal if $X$ is normal. 
Thus it suffices to prove that $Z$ is not normal if $e'(w)>1$. 

Recall from~\eqref{eq:sOdecomp} in the proof of Theorem~\ref{th:group_S_w} that 
$A:=\sigma^*(\Oh(Z)) = \bigoplus_{s\in \E'(w)} \C T^{sb(w)}$.
Suppose that $e'(w)>1$.  We have $e:=e'(w)-1\not\in \E'(w)$ since
$e'(w)$ is defined as the minimal positive element of $\E'(w)$. 
Theorem~\ref{th:group_S_w} states that $\E'(w)$  
generates the group $\Z$. So we have $ke\in \E'(w)$ for a sufficiently large $k\in\N$
(cf.\ Section~\ref{se:num-sg})  
and hence $T^{keb(w)} \in A$. 
The element $T^{eb(w)}$ satisfies the monic equation 
$x^k - T^{keb(w)}=0$ with coefficients in~$A$, 
hence it lies in the integral closure of $A$.

On the other hand, there are $s_1,s_2\in \E'(w)$ such that $e=s_1-s_2$, 
since $\E'(w)$ generates the group~$\Z$.
It follows that 
$T^{eb(w)} = T^{s_1b(w)}/T^{s_2b(w)}$ is in the quotient field of $A$, but not in $A$. 
Therefore, $A$ is not integrally closed.

Finally, $\Oh(X)$ is not a valuation ring since it is known that 
valuation rings are integrally closed~\cite[Thm.~10.3]{mats:86}. 
\qed

\subsection{Proof of Theorem~\ref{th:group_S_w-T}} 
\label{se:Editto} 

Recall the character $\chi\colon G\to\Ct$ from \eqref{eq:def-chi}, 
where $G =\GL_m^3/K$ and define the subgroup $G_s:=\ker\chi = \SL_m^3/K$. 
We have a group embedding 
$$
 \iota\colon\Ct\to G,\, t\mapsto (t\id_m,\id_m,\id_m) \bmod K .
$$
(Note that 
$(t\id_m,\id_m,\id_m) \equiv (\id_m,t\id_m,\id_m) \equiv (\id_m,\id_m,t\id_m) \bmod K$.)
The image $\iota(\Ct)$ is a normal subgroup of $G$ and  
$G$ is the product of $\iota(\Ct)$ and the subgroup $G_s$.  
The group $G$ acts on $W:=\ot^3\C^m$ and we have 
$\iota(t) w = tw $ for $t\in\Ct$ and $w\in W$
(compare \eqref{eq:iota(t)}). 
We form the GIT-quotient of $X:=\ol{Gw}$ by the action of $G_s$, 
obtaining a $\Ct$-equivariant morphism $\pi\colon X\to Z$. 
Lemma~\ref{le:sigma_image} holds with essentially the same proof. 
The proof of Theorem~\ref{th:group_S_w-T} proceeds along the same lines as for 
Theorem~\ref{th:group_S_w}, using the following result, which is similar to 
Lemma~\ref{le:stab-piw}. 

\begin{lemma}\label{le:stab-piw-tensor}
The stabilizer of $p:=\pi(w)$ satisfies 
$\{t\in\Ct \mid tp= p\} = \mu_{m a(w)}$. 
\end{lemma}

\begin{proof}
Suppose that $\pi(tw)=\pi(w)$. Then, since $\pi$ is a GIT quotient, 
the orbit closures of $tw$ and $w$ intersect; cf.~\cite[II.\S3.2]{kraf:84}. 
Since $G_sw$ is closed, this means that $tw\in G_sw$. 
Therefore, there exists $g_s\in G_s$ such that $tw=g_s^{-1}w$, 
hence $g_s tw = w$.  As before, we have 
$g_s t w= g_s\iota(t)w$, hence $g_s \iota(t)\in \stab(w)$. 
By the definition of the period~$a=a(w)$, we obtain 
$$
1 = \chi(g_s \iota(t))^a = \chi(\iota(t))^a = (t^{m})^a  =  t^{ma}.
$$
Therefore,
$\{t\in\Ct \mid tp= p\} \subseteq \mu_{ma}$. 
The reverse inclusion is proven analogously.
\end{proof}

\section{Appendix}\label{appendix}

We believe the results here should be known, but they are hard to locate 
in the literature. We are therefore brief. 

\subsection{Plethysms}

\begin{proposition}\label{pro:plethupperbound}
Let $D$ be odd and let $\la$ be a partition of $Dd$ into at most $m$ parts.
Let $\mult_\la(\Oh(\Sym^D\IC^m)_d)$ denote the multiplicity of the irreducible $\GL_m$-representation of type $\la^*$ in $\Oh(\Sym^D\IC^m)_d$.
A set $S$ of subsets of $\{1,\ldots,\la_1\}$ is called \emph{of type $\la$ in degree $d$} if
\begin{itemize}
 \item The cardinality $|S|$ of $S$ is $d$,
 \item each element $s \in S$ has cardinality $D$,
 \item each number $i$ occurs in exactly $\la^t_i$ of the subsets $s$ of $S$, where $\la^t_i$ is the length of the $i$th column of~$\la$.
\end{itemize}
Let $q_\la(d)$ denote the number of sets $S$ of type $\la$ in degree $d$. Then
\[
\mult_\la(\Oh(\Sym^D\IC^m)_d) \leq q_\la(d).
\]
\end{proposition}
\begin{proof}
Using \cite[Fact~6.1]{MM:14} we know that
the plethysm coefficient $\mult_\la(\Oh(\Sym^D\IC^m)_d)$ equals $\mult_{\la^t}(\bigwedge^d\bigwedge^D\IC^m)$ if $D$ is odd,
where $\la^t$ is the partition to the transposed Young diagram of~$\la$.
This multiplicity is bounded from above by the dimension of the $\la^t$-weight space $(\bigwedge^d\bigwedge^D\IC^m)^{\la^t} \subseteq \bigwedge^d\bigwedge^D\IC^m$:
\[
\textstyle \mult_{\la^t}(\bigwedge^d\bigwedge^D\IC^m) \leq \dim (\bigwedge^d\bigwedge^D\IC^m)^{\la^t}.
\]
The right hand side is precisely the number of cardinality $d$ sets of cardinality $D$ subsets of $\{1,\ldots,D\}$
such that each number $i$ occurs in exactly $\la^t_i$ subsets.
\end{proof}
For $\SL_m$-invariants this specializes as follows.
\begin{corollary}\label{cor:plethupperbound}
Let $D$ be odd and let $Dd$ be divisible by $m$.
The dimension of the $\SL_m$-invariant space in $\Oh(\Sym^D\IC^m)_d$
is bounded from above
by the number of cardinality $d$ sets of cardinality $D$ subsets of $\{1,\ldots,dD/m\}$
such that each number occurs in precisely $m$ of the subsets.
\end{corollary}

We state without proof that Proposition~\ref{pro:plethupperbound} can be strengthened as follows
by factoring out the relations among the weight vectors:
\begin{proposition}
Let $\la$ be a partition of $Dd$ into at most $m$ parts.
For $D$ odd,
the plethysm coefficient $\mult_\la(\Oh(\Sym^D\IC^m)_d)$ is bounded from above by the number of semistandard tableaux
\begin{itemize}
 \item of shape $\la$
 \item in which each number from the set $\{1,\ldots,D\}$ appears exactly $d$ times
 \item and in which there do not exists two numbers $i$ and $j$
with the property that they occur in precisely the same $d$ columns.
\end{itemize}
\end{proposition}

\subsection{Stabilizers}

Let $\mu_d\subseteq\Ct$ denote the group of $d$th roots of unity. 
We first consider binary forms.

\begin{proposition}\label{stab:bin-forms}
\begin{enumerate}
\item For a generic $w\in \Sym^3\C^2$ we have 
$\stab(w)\simeq \mu_3\rtimes S_2$. Hence  $a(3,2)=6$ and $a'(3,2)=2$.  

\item For a generic $w\in \Sym^4\C^2$ we have 
$\stab(w) \simeq C_2\ti C_2$. Hence $a(4,2)=2$ and $a'(4,2)=1$. 

\item Let $D\ge 5$. The stabilizer of a generic $w\in \Sym^D\C^2$ 
is trivial, hence $a'(D,2)=1$.

\end{enumerate}
\end{proposition}

\begin{proof}
Suppose $D=3$.
It is know that all $w\in\Sym^3\C^2$ with three distinct zeros are equivalent to $v=X^3+Y^3$. 
Using Proposition~\ref{pro:stab-X1-Xn-PS} we get 
$\stab(v)\simeq \mu_3\rtimes S_2$ and the first assertion follows.

Suppose $D=4$.
Almost all $w\in\Sym^4\C^2$ are equivalent to a form 
$F_\mu := X^4 + \mu X^2 Y^2 + Y^4$. 
(In order to show this, check that the derivative of the 
orbit map 
$\GL_2\ti\C \to \Sym^4\C^2, (g,\mu)\mapsto g F_\mu$ 
is surjective if $\mu\ne \pm 2$.) 
The Hessian of $F_\mu$ equals 
$$
 H(F_\mu) = 24\mu(X^4+Y^4) + 12(12-\mu^2)X^2Y^2 .
$$
It is easy to check that 
$\spann\{F_\mu, H(F_\mu)\} = \spann \{ X^4+Y^4, X^2Y^2\}$ 
if $\mu\ne \pm 2$. 
Suppose that $g\in\stab(F_\mu)$ satisfies $\det g=1$.
By \eqref{eq:detH-stab},
$g$ stabilizes $H(F_\mu)$ as well, hence $g$ stabilizes $X^4+Y^4$.
Proposition~\ref{pro:stab-X1-Xn-PS} 
implies that $\stab(F_\mu)$ is generated by 
$\begin{bmatrix} 0& 1 \\ 1& 0\end{bmatrix}$ 
and the matrices 
$\begin{bmatrix} t_1 & 0 \\ 0 & \pm t_1^{-1}\end{bmatrix}$,
where $t_1 \in \mu_4$.
This implies $\stab(F_\mu) \simeq C_2\ti C_2$, $a(F_\mu)=2$
and the second assertion follows.

We omit the proof in the case $D\ge 5$. 
\end{proof}

Now we consider ternary forms.

\begin{proposition}\label{stab:tern-forms}
\begin{enumerate}
\item A generic $w\in \Sym^3\C^3$ satisfies 
$\stab(w)\simeq \mu_3 \rtimes S_3$. 
Hence $a'(3,3)=2$. 

\item For a generic $w\in \Sym^4\C^3$ we have 
$\stab(w)\simeq C_2$. Hence $a'(4,3)=2$ .

\item Let $D\ge 5$. The stabilizer of a generic $w\in \Sym^D\C^3$ is trivial, 
hence $a'(D,3)=1$.

\end{enumerate}
\end{proposition}

\begin{proof}
Case $D=3$:  
A generic ternary form is in the same orbit as a smooth cubic in Hesse normal form 
$v:= X^3 + Y^3 + Z^3 + \lambda XYZ$ where $\lambda^3 \ne -27$, 
cf.\ \cite[\S3.1.2]{dolgachev:12}. 
In order to determine the stabilizer, we compute the Hessian 
$H(v) = (216 + 2\lambda^3)XYZ - 6\lambda^2(X^3+ Y^3+Z^3)$ 
and check that 
$\spann\{v,H(v)\} = \spann\{ X^3+Y^3+Z^3, XYZ\}$ 
due to $\lambda^3 \ne -27$. 
If $g\in\stab(v)$ such that $\det(g)=1$, then 
$g\in\stab H(v)$ by \eqref{eq:detH-stab},
hence $g\in\stab H(XYZ)$. 
The assertion follows with Proposition~\ref{pro:stab-X1-Xn-PS}.

In the case $D=4$ we can use a normal form in 
\cite[\S6.5.2, Table~6.1, p.~266]{dolgachev:12}. 

In the case $D\ge 5$ we can argue similarly
as in Matsumura and Monsky~\cite{mat-mons:63}.
However, some care has to be taken since their 
proof actually contains an error, 
which fortunately does not affect its statement 
under the hypothesis $D >2$ and $m>3$. 
Applying the same technique
in the case $D \ge 5$ and $m=3$,
and correcting the error, yields the assertion.
(We note that Matsumura and Monsky's~\cite{mat-mons:63}
uncorrected proof technique yields the wrong answer in the 
case $D = 4$ and $m=3$.)
\end{proof}

  {\small 
  }
\end{document}